\documentclass[12pt]{article}
\usepackage{sectsty}
\sectionfont{\fontsize{12}{10}\selectfont}
\usepackage{bm}
\usepackage{geometry}
\usepackage[dvips]{epsfig}
\usepackage{mathtools}
\usepackage{amssymb,amsmath,amsthm,latexsym,isomath}
\usepackage{color}
\usepackage[dvipsnames]{xcolor}
\usepackage{stmaryrd}
\usepackage{mathrsfs} 
\usepackage{comment}
\usepackage{enumerate}
\usepackage{soul}
 \usepackage{cancel}
\usepackage[perpage]{footmisc} 
\usepackage{hyperref}
\theoremstyle{plain}
\newtheorem{thm}{Theorem}[section]
\newtheorem{lemma}[thm]{Lemma}
\newtheorem{cor}[thm]{Corollary}

\newtheorem{example}{Example}

\newtheorem{fact}[thm]{Fact}

\theoremstyle{definition}
\newtheorem{rk}{Remark}
\theoremstyle{remark}

\numberwithin{equation}{section} 
\numberwithin{defn}{section}
\numberwithin{rk}{section}

\newcommand{\rsim}[1]{\textcolor{cyan}{#1}}
\newcommand \revi[1]{{\color{red}#1}}
\newcommand\norm[1]{\left\|#1\right\|}
\newcommand\ip[2]{\left(#1,#2\right)}
\newcommand\abs[1]{\lvert#1\rvert}

\newcommand{\tforall}{\text{ for all }}

\usepackage[titletoc,title]{appendix}
\usepackage{titlesec}
\titlelabel{\thetitle.~}

\usepackage{graphicx}
\graphicspath{%
    {converted_graphics/}
}

\usepackage{blindtext} 
\usepackage{fancyhdr}
\usepackage{algorithm}
\usepackage{algpseudocode}
\usepackage{titling}
\pretitle{\begin{center}\large\scshape\bfseries}
\posttitle{\par\end{center}}
\usepackage{amsmath}               
  {
      \theoremstyle{plain}
      \newtheorem{assumption}{Assumption}
  }
\preauthor{\begin{center}\scshape \begin{tabular}[t]{l r}}
\postauthor{&}
\predate{\scshape}
\postdate{\end{tabular}\par\end{center}}
\title{\sc A Multiscale Method for the Heterogeneous Signorini Problem}
\author{Xin Su\thanks{Department of Mathematics, Texas A\&M University, College Station, TX 77843, USA},~ Sai-Mang Pun\thanks{Department of Mathematics, Texas A\&M University, College Station, TX 77843, USA}
}
\begin{document}
\maketitle

\begin{abstract}
    In this paper, we develop a multiscale method for solving the Signorini problem with a heterogeneous field. The Signorini problem is encountered in many applications, such as hydrostatics, thermics, and solid mechanics. 
    It is well-known that numerically solving this problem requires a fine computational mesh, which can lead to a large number of degrees of freedom. 
    The aim of this work is to develop a new hybrid multiscale method based on the framework of the generalized multiscale finite element method (GMsFEM).
    The construction of multiscale basis functions requires local spectral decomposition.
    Additional multiscale basis functions related to the contact boundary are required so that 
    our method can handle the unilateral condition of Signorini type naturally. A complete analysis of the proposed method is provided and a result of the spectral convergence is shown. Numerical results are provided to validate our theoretical findings. 
\end{abstract}


\section{Introduction}
In this paper, we develop a multiscale method for Signorini’s problem in heterogeneous media. The Signorini problem is a classical unilateral problem, which was first studied by Fichera in the theory of elasticity \cite{fichera1964problemi} and later, Lions and Stampacchia \cite{lions1967variational} built the foundations of the theory. 
The Signorini problem occurs in many applications in science and engineering, including hydrostatics, thermics, and solid mechanics. We refer the readers to \cite{haslinger1996,kinderlehrer2000introduction} for a comprehensive review of the main unilateral contact models. It is well-known that the unilateral problem is characterized by a variational inequality set on a closed convex functional cone. 
There are a lot of works in discussing how to impose the unilateral conditions based on the well-posedness of the discrete formulation and the accuracy of the approximation algorithms (see \cite{ haslinger1996,tremolieres2011numerical,kikuchi1988contact,haslinger1981contact}).  
In  \cite{bend}, the authors revisited three different types of hybrid finite element methods for the Signorini problem. In our paper, we choose to use a hybrid variational formulation and apply the unilateral conditions to the displacement and the stress on the contact zone separately.

Model reduction technique for solving standard PDE in variational equality format has been extended to variational inequality recently. In \cite{haasdonk2012reduced}, the authors generalized the reduction technique in  \cite{rozza2008reduced}, which constructs the reduced basis by combining the greedy algorithm and a posterior error indicator, to solve parametrized variational inequalities. 
Applications to contact mechanics, optimal control, and some improvements based on this method have been proposed in \cite{benaceur2020reduced,zhang2016slack,negri2013reduced,balajewicz2016projection} and the references therein. 
Another family of model reductions applied to variational inequalities is based on the proper orthogonal decomposition (POD) methodology (see \cite{fauque2018hybrid,krenciszek2014model} and the references therein). 
In \cite{fauque2018hybrid}, the authors proposed an extension of the hyper-reduction method based on a reduced integration domain to frictionless contact problems which are written in mixed formulations. 
To be specific, they applied the POD method to reduce the number of the primal bases (for the displacement) and built a reduced integration domain based on the discrete empirical interpolation method (DEIM) to naturally reduce the number of the dual bases (for the contact forces).

However, many problems arising from physics and engineering applications involve multiple scales feature. 
In this research, we aim at approximating the solution of a high-contrast multiscale Signorini problem. 
The solution technique for problems involving multiple scales and high contrast requires high resolution. 
This results in fine-scale problems with high degrees of freedom, which are expensive to solve.
To reduce the computational cost, we need some types of model reduction. 
Our approach here uses the Generalized Multiscale Finite Element Method (GMsFEM), which was recently introduced in \cite{efendiev2013generalized}. 
The GMsFEM is a systematic technique for constructing multiscale basis functions, 
which are designed to capture the multiscale features of the solution. 
The multiscale basis functions contain information about the scales that are smaller than the coarse scale related to the basis functions (see \cite{efendiev2013generalized}).  
The GMsFEM is widely used in solving problems whose variational forms are of the types of equations. 
One may see \cite{chung2015mixed,chung2014generalized,chung2016generalized,gao2015generalized,efendiev2014generalized} and the references therein for more details. 

The main contribution of our work is as follows.
We develop a novel multiscale method based on the GMsFEM for solving the Signorini problem for the first time. 
Moreover, the work presented here can serve as a foundation for more general study of multiscale method for variational inequalities. 
The construction of multiscale space starts with a collection of snapshots. Then, a spectral decomposition of the snapshot space is performed to reduce the dimension of the fine-scale space.
Specifically, we consider a class of local spectral problems motivated by the analysis and we select the dominant modes corresponding to the first few smallest eigenvalues. 
To ensure the positivity of the displacement along the contact boundary, we introduce a set of additional multiscale basis functions, which are harmonic extension with appropriate essential boundary condition.
A rigorous analysis of the proposed method is presented.
In particular, we show a result of spectral convergence depending on the smallest eigenvalue whose eigenfunction is excluded in the local multiscale space.
We present numerical results for two different heterogeneous permeability fields and show that we can get accurate approximations with fewer degrees of freedom using the proposed method. The first case we consider contains small inclusions while the other one contains several high contrast channels.
Both examples show a fast convergence of the GMsFEM.

The paper is organized as follows. In Section \ref{sec:prelim}, we state the problem and the finite-element approximation on a fine grid. In Section \ref{sec:gmsfem}, the construction of multiscale basis functions is presented. In Section \ref{sec:analysis}, we present the error analysis and show the convergence for our method. In Section \ref{sec:numerics}, we present numerical experiments to show the effectiveness and the efficiency of our method. The paper ends in Section \ref{sec:conclusion} with a conclusion.

\section{Preliminaries} \label{sec:prelim}
In this section, we will present the model problem considered in this work. 
In Section \ref{subsec:problem}, we present the Signorini problem with a high-contrast permeability field. 
In Section \ref{subsec:functional}, we introduce the functional spaces that would be used in this work. 
For numerical approximation, we define the fine and coarse grids in Section \ref{subsec:partition}. 
This section ends with a fine-grid approximation scheme in Section \ref{subsec:finegrid}.

\subsection{Problem Setting} \label{subsec:problem}
Let $\Omega\subset \mathbb{R}^d$ ($d \in \{ 2, 3\}$) be a bounded polygonal Lipschitz domain. We consider the following Signorini problem over the computational domain $\Omega$: For a given source function $f \in L^2(\Omega)$, find a function $u$ such that it solves the following problem
\begin{equation}
\label{eqn:model}
\begin{cases}
    -\nabla \cdot (\kappa(x)\nabla u)=f &\text{ in } \Omega  , \\
     u =0 &\text{ on }\Gamma_D,\\
     \kappa(x)\nabla u \cdot \boldsymbol{ n}=0 &\text{ on } \Gamma_N,\\
     u\geq 0, ~ \kappa(x)\nabla u\cdot \boldsymbol{ n}\geq 0, ~ u (\kappa(x)\nabla u\cdot \boldsymbol{ n})=0 &\text{ on }\Gamma_C.
\end{cases}
\end{equation}
Here, $\kappa(x)$ is a heterogeneous permeability field with high contrast and $\boldsymbol{ n}$ is the unit outward normal vector field to the boundary $\partial \Omega$. We assume that there exist two positive constants $\kappa_0$ and $\kappa_1$ such that $0<\kappa_0\leq\kappa(x)\leq \kappa_1<\infty$  {for almost all $x \in \Omega$ and the ratio $\kappa_1/ \kappa_0$ could be very large}.
The boundary $\partial \Omega$ is a union of three non-overlapping portion $\Gamma_D,$ $\Gamma_N$, and $\Gamma_C$. 
The vertices of $\Gamma_C$ are $\{\boldsymbol{c}_1, \boldsymbol{c}_2\}$ and those of $\Gamma_D$ are $\{\boldsymbol{c}_1^\prime, \boldsymbol{c}_2^\prime\}$. See Figure \ref{fig:Triangulation} for a graphical illustration. 
The Neumann and unilateral boundary conditions are understood in the sense of distributions on $\partial \Omega$. 

\subsection{Functional Framework}\label{subsec:functional}

The functional framework well suited to solve the Signorini problem  \eqref{eqn:model} with heterogeneous permeability field consists in working with the subspace $H^1_0(\Omega, \Gamma_D)$ of the Sobolev space $H^1(\Omega)$ made of functions that vanish at $\Gamma_D$. That is, 
$$ H_0^1(\Omega, \Gamma_D) := \left \{ v \in H^1(\Omega): v|_{\Gamma_D} = 0 \right \}.$$

At the outset of the variational formulation of this problem, the unilateral contact condition enters the variational formulation through the introduction of the closed convex set 
\begin{equation}
\label{Q5}
K(\Omega):=\left\{v \in H_{0}^{1}\left(\Omega, \Gamma_{D}\right):  v|_{ \Gamma_C} \geq 0 \text { a.e. in } \Gamma_C \right\}.
\end{equation}
The primal variational principle for \eqref{eqn:model} produces the variational
inequality:  find $u\in K(\Omega)$ such that 
\begin{equation}
\label{Q6}
a(u, v-u) \geq L(v-u) \quad \tforall  v \in K(\Omega).
\end{equation}
The bilinear form $a: H^1(\Omega) \times H^1(\Omega) \to \mathbb{R}$ and the linear functional $L \in (H^1(\Omega))'$ are defined by
$$
\begin{aligned} 
a(u, v) &:=\int_{\Omega} \kappa(x)\nabla u \cdot\nabla v ~dx   \quad&&\text{ for all }  u, v \in H^1(\Omega),
\\ L(v) &:=\int_{\Omega} f v d x  \quad &&\text{ for all }  v \in H^1(\Omega).
\end{aligned}
$$
Let us first introduce the notations for norms and inner products that will be used in this work. We denote the weighted $L^2$ inner product and the associated norm on $\Omega$ as
$$\ip{u}{v}_{L^2_\kappa(\Omega)}:=\int_\Omega \kappa uv ~ dx, ~~\norm{v}_{L^2_\kappa(\Omega)} :=\sqrt{\ip{v}{v}_{L^2_\kappa(\Omega)}}.$$ 

We remark that we write $\norm{\cdot}_{L^2(\Omega)}$ and $(\cdot,\cdot)$ to denote the $L^2$ norm and inner product when $\kappa \equiv 1$. 
Next, we denote the weighted $H^1$ and the semi-norm weighted $H^1$ norm as

$$ |v|_{H_\kappa^1(\Omega)} := \sqrt{a(v,v)},~~ \norm{v}_{H^1_\kappa(\Omega)} := \left ( \norm{v}_{L^2_\kappa(\Omega)}^2 + |v|_{H_\kappa^1(\Omega)} ^2 \right  )^{1/2}.$$
Here, we denote $\alpha = (\alpha_1, \cdots, \alpha_d) \in \mathbb{N}^d$ a multi-index with $\abs{\alpha} := \sum_{i=1}^d \alpha_i$ and 
$$ D^{\alpha} v := \frac{\partial^{\abs{\alpha}} v}{\partial_{x_1}^{\alpha_1} \cdots \partial_{x_d}^{\alpha_d}}.$$
We define $$L^2_\kappa(\Omega):=\left \{u:\Omega\to \mathbb{R}  \text{ measurable such that } \|u\|_{L^2_\kappa(\Omega)}<\infty \right \},$$ 
$$H^1_\kappa(\Omega):=\left \{u\in L_\kappa^2(\Omega): \partial^\alpha_x u\in L^2_\kappa(\Omega) \tforall \alpha \text{ with }  |\alpha| \leq 1  \right \},$$ 
On the other hand, one has to introduce appropriate functional spaces in order to describe the unilateral boundary condition defined on $\Gamma_C$. We define 
$$L_\kappa^2(\Gamma_C):=\left \{v:\Gamma_C\to \mathbb{R} \text{ measurable such that }  \int_{\Gamma_C}\kappa v^2~ d\Gamma<\infty \right \}$$ endowed with the norm $\|v\|_{L_\kappa^2(\Gamma_C)}:=\left(\int_{\Gamma_C}\kappa v^2~ d\Gamma\right)^{\frac{1}{2}}$ and 
$$H^1_\kappa(\Gamma_C):=\left \{ v\in L_\kappa^2(\Gamma_C) \text{ with } \frac{dv}{ds} \in L_\kappa^2(\Gamma_C) \right \}$$ endowed with the norm $\|v\|_{H^1_{\kappa}(\Gamma_C)}:=\left(\int_{\Gamma_C} \kappa v^2 +\kappa (\frac{dv}{ds})^2~d\Gamma\right)^{\frac{1}{2}} $ for the case of $d=2$. For three dimension, see Appendix \ref{3dbdnorm}.
The space $H^{\frac{1}{2}}_\kappa(\Gamma_C)$ is defined via the real interpolation between $L^2_\kappa(\Gamma_C)$ and $H^1_\kappa(\Gamma_C)$. 
\begin{rk}
Here, $\frac{dv}{ds}$ is the derivative of $v$ along the boundary $\Gamma_C$ since we can view  $v$ restricted on $\Gamma_C$ as a function of a single variable $s$.
\end{rk}

The space $H^{-\frac12}_\kappa(\Gamma_C)$ is defined as the dual space of $H^{\frac{1}{2}}_\kappa(\Gamma_C)$ endowed with the norm $$\|v\|_{H^{-\frac{1}{2}}_\kappa(\Gamma_C)}:=\sup_{w\in H^{\frac{1}{2}}_\kappa(\Gamma_C)}\frac{b(v,w)}{~~~~\|w\|_{H^{\frac{1}{2}}_\kappa(\Gamma_C)}}.$$

In order to have weighted Poincar\'{e} inequalities, we need to make the following assumption on $\Omega$ and $\kappa$. 
\begin{assumption}\label{assumption}
(Structure of $\Omega$ and $\kappa$). Let $\Omega$ be a domain with a $C^{1,\alpha}(0<\alpha<1)$ boundary $\partial \Omega $, and $\{\Omega_i\}_{i=1}^m\subset \Omega$ be m pairwise disjoint strictly convex open subsets, each with $C^{1,\alpha}$ boundary $\partial \Omega_i$, and denote $\Omega_0=\Omega\backslash \overline{\cup_{i=1}^m \Omega_i}$. Let the permeability coefficient $\kappa$ be piecewise regular function defined by 
\begin{equation}
    \kappa=\begin{cases}
    \eta_i(x) &\text{ in } \Omega_i,\\
    1 &\text{ in } \Omega_0.
    \end{cases}
\end{equation}

Here, $\eta_i\in C^{\mu}(\overline{\Omega_i})$ with $\mu\in(0,1)$ for $i=1,\cdots, m$. Denote $\eta_{\text{min}}:=\min_i\{\eta_i\}\geq 1$ and $\eta_{\text{max}}:= \max_i\{\eta_i\}$.
\end{assumption}

The following weighted Poincar\'{e} inequality is presented in \cite{pechstein2013weighted} under Assumption \ref{assumption}. One may also refer to \cite{li2019convergence,galvis2010domain}.  In this work, we assume that $\kappa$ is a piecewise constant function.
\begin{thm}
Under Assumption \ref{assumption}, the following weighted Poincar\'{e} inequality holds: 
\begin{eqnarray}
\int_\Omega \kappa \abs{v}^2 dx \leq C_{\text{poin}} \int_\Omega \kappa \abs{\nabla v}^2 dx \quad \tforall v \in K(\Omega),
\label{ineq:wpi}
\end{eqnarray}
where the constant $C_{\text{poin}}>0$ is independent of the coefficient $\kappa$. 
\end{thm}

\begin{rk}
Although the weighted Poincar\'{e} inequality holds under Assumption \ref{assumption}, which is more general, we assume that $\kappa$ is a piecewise constant function in this work.
\end{rk}


With the weighted Poincar\'{e} inequality \eqref{ineq:wpi} and appropriate assumption on $\kappa$, it can be shown that $a(\cdot,\cdot)$ is a continuous, coercive, and symmetric bilinear form on $K(\Omega)$. In other words, there is a positive constant $\alpha >0$  such that 
$$ \alpha \norm{v}_{H^1_\kappa(\Omega)}^2 \leq a(v,v) \quad \text{and} \quad \abs{a(u,v)} \leq  \norm{u}_{H^1_\kappa(\Omega)}\norm{v}_{H^1_\kappa(\Omega)}$$
for any $u \in K(\Omega)$ and $v \in K(\Omega)$. 
Applying Stampacchia's theorem \cite{lions1967variational}, the variational inequality \eqref{Q6} has a unique solution in $K(\Omega)$ that depends continuously on the data $f$.

The variational inequality \eqref{Q6} is equivalent to the following constrained minimization problem:  {find $u \in K(\Omega)$} such that 
\begin{equation}
J(u)=\inf _{v \in K(\Omega)} J(v), \quad \text {where} \quad J(v):=\frac{1}{2} a(v,v)-L(v).
\label{eqn:con_j}
\end{equation}
The constrained minimization problem above can be reformulated as a saddle point problem, and hence $u$ is the first component of the saddle point $(u,\phi)$ of the Lagrangian 
\begin{equation}
\label{Q7}
\mathcal{L}(v,\psi):=\frac{1}{2}a(v,v)-L(v)-b(\psi,v)
\end{equation}
defined on $H^1_0(\Omega,\Gamma_D)\times M(\Gamma_C),$ where 
$$
M(\Gamma_C):= \left \{\psi\in H^{-\frac{1}{2}}(\Gamma_C): \psi\geq0 \right \} \quad \text{and} \quad b(\psi,v):=\langle\kappa \psi,v\rangle_{\frac{1}{2},\Gamma_C}.
$$
Here, $\left \langle \cdot, \cdot \right \rangle_{\frac{1}{2}, \Gamma_C}$ denotes the duality pair over $H^{-\frac{1}{2}}(\Gamma_C)$ and $H^{\frac{1}{2}} (\Gamma_C)$. 
The non-negativity of a distribution $\psi\in H^{-\frac{1}{2}}(\Gamma_C)$ is understood in the sense that $b (\psi, \chi)\geq 0$ for any $\chi\in H^{\frac{1}{2}}(\Gamma_C)$ with $\chi\geq 0$. 

The properties of saddle point problem \cite{haslinger1996} leads to the following hybrid variational system: find $(u,\varphi)\in H^1_0(\Omega, \Gamma_D)\times M(\Gamma_C)$ such that
\begin{eqnarray}
\begin{aligned}
a(u, v)-b(\varphi, v) &=L(v) &\quad \tforall v \in H_{0}^{1}\left(\Omega, \Gamma_{D}\right),  \\
b(\psi-\varphi, u) &\geq 0 &\quad \tforall \psi \in M(\Gamma_{C}).  
\end{aligned}
\label{Q8}
\end{eqnarray}
This hybrid problem \eqref{Q8} has a unique solution which satisfies the following stability condition
\begin{equation}
\|u\|_{H^1_\kappa(\Omega)}+\|\varphi\|_{H^{-\frac{1}{2}}_\kappa\left(\Gamma_{C}\right)} \leq C\norm{\kappa^{-1/2}f}_{L^{2}(\Omega)}.
\end{equation}
Existence  {and uniqueness} to the hybrid problem \eqref{Q8} as well as equivalence of the constrained minimization problem \eqref{eqn:con_j} and the hybrid formulation \eqref{Q8} {have} been proven, e.g., in Haslinger et al. \cite{haslinger1996}.

\subsection{Partitions of Domain}\label{subsec:partition}
In this section, we introduce partition of domain and boundary for numerical discretization. We start with introducing the notions of fine and coarse grids.

Let $\mathcal{T}^H = \{ K_i \}_{i \in \mathcal{I}}$ be a usual conforming partition of the computational domain $ {\Omega}$, which contains a collection of simplices (triangles, quadrilaterals, or tetrahedrals) with mesh size $H>0$ such that 
$$ H := \max_{i \in \mathcal{I}} \max_{x, y \in K_i} \abs{x-y}_{\mathbb{R}^d}.$$
Here, $\mathcal{I}$ is a general finite index set and $\abs{\cdot}_{\mathbb{R}^d}$ denotes the Euclidean metric in $\mathbb{R}^d$. 
Each coarse element $K \in \mathcal{T}^H$ is further partitioned into a connected union of fine-grid blocks, which are conforming across coarse-grid edges. 
The fine-grid partition is denoted by $\mathcal{T}^h$, which by definition is a refinement of $\mathcal{T}^H$, with mesh size $h$ much less than $H$. 
We denote $\tau$ and $K$ the generic elements in the fine-grid and the coarse-grid partitions, respectively. 
The sets of the grid nodes in $\mathcal{T}^h$ and $\mathcal{T}^H$ are denoted as $\Xi_h$ and $\Xi_H$, respectively. 
Let $N$ be the number of coarse grid nodes in $\Xi_H$. 
We number first the nodes lying on $\Gamma_C$, say from $1$ to $N_0$, and the remaining nodes from $N_0+1$ to $N$. 
We write $\Xi_H:=\{\boldsymbol{a_i}: 1\leq i\leq N\}$. 
For a coarse node $\boldsymbol{a_i} \in \Xi_H$, we define a local domain $\omega_i$ associated with $\boldsymbol{a_i}$ as follows: 
$$ \omega_i := \bigcup \{ K \in \mathcal{T}^H : \boldsymbol{a_i} \in K \}.$$
See Figure \ref{fig:Triangulation} for a graphical illustration of the mesh and coarse neighborhood. 

\begin{figure}[htp!]
\begin{center}
\includegraphics[scale=0.2]{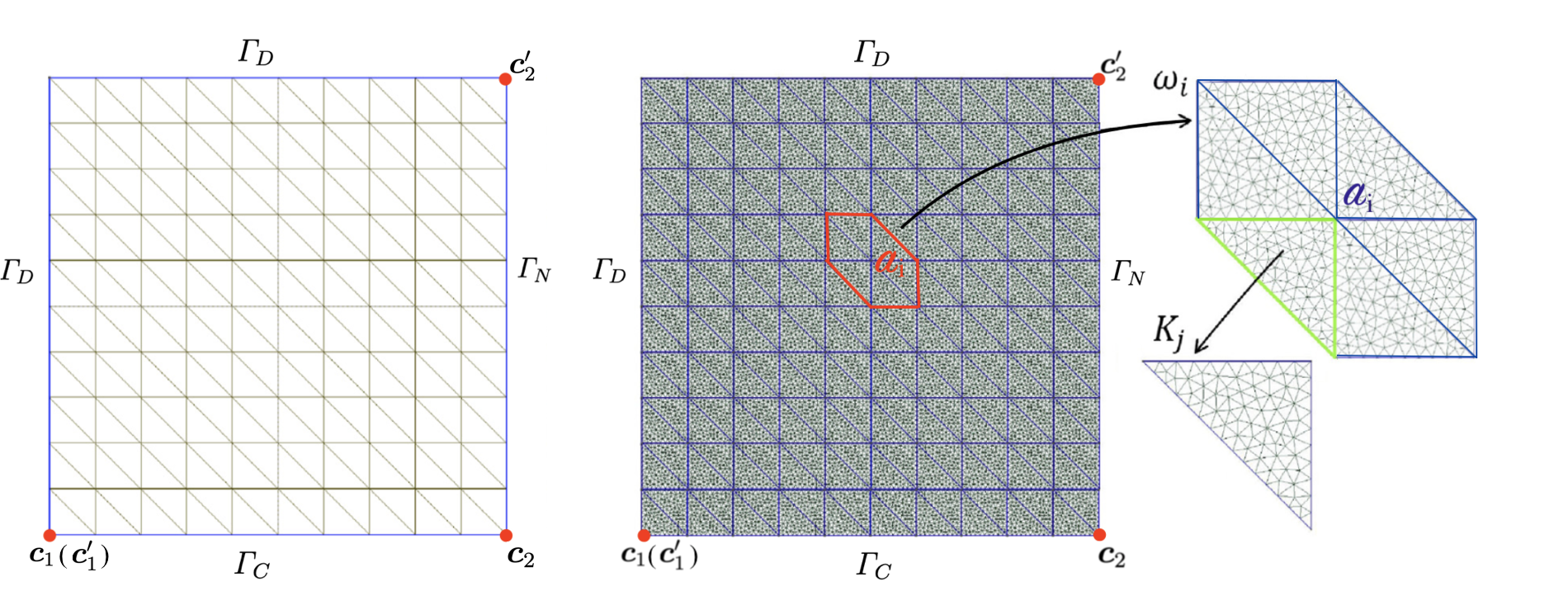}
\caption{Coarse grid and local domain $\omega_i$ with $K_j$}
\label{fig:Triangulation}
\end{center} 
\end{figure}

We assume that the partitions $\mathcal{T}^h$ and $\mathcal{T}^H$ are $\mathcal{C}^0$-regular in the classical sense \cite{ciarlet2002finite}. 
Moreover, $\mathcal{T}^h$ and $\mathcal{T}^H$ are built in such a way that the set $\{\boldsymbol{c}_1, \boldsymbol{c}_2,  \boldsymbol{c}_1^\prime,  \boldsymbol{c}_2^\prime\}$ is included in $\Xi_h$ and $\Xi_H$.  
Due to the $ {\mathcal{C}^0}$-regularity hypothesis, the contact boundary $\Gamma_C$ inherits regular meshes. 
Denote $\mathcal{T}^h_C$ characterized by the subdivision $\left(\boldsymbol{x}_{i,h}^{C}\right)_{1\leq i \leq i^{*}_h}$ inherited from $\mathcal{T}^h$ with 
$\boldsymbol{x}_{1,h}^{C}=\boldsymbol{c}_{1}$, $\boldsymbol{x}_{i_h^{*},h}^{C}=\boldsymbol{c}_{2}$, and $t_{i,h}:=\left(\boldsymbol{x}_{i,h}^{C}, \boldsymbol{x}_{i+1,h}^{C}) \right)_{1 \leq i \leq i^{*}_h-1}$ are its elements. 
Similarly, we denote $\mathcal{T}^H_C$ characterized by the subdivision $\left(\boldsymbol{a_i}\right)_{1\leq i\leq N_0}$ inherited from $\mathcal{T}^H$ with $\boldsymbol{a}_1=\boldsymbol{c}_{1}$, $\boldsymbol{a}_{N_0}=\boldsymbol{c}_{2}$, and $t_{i,H}:=\left(\boldsymbol{a}_{i}, \boldsymbol{a}_{i+1}\right)_{1\leq i \leq N_0-1}$ are its elements. 

\subsection{Fine-grid Approximation}\label{subsec:finegrid}
In this section, we introduce the fine-grid approximation of the hybrid variational formulation \eqref{Q8}. 
For any set $S \subset \overline{\Omega}$, we denote $\mathcal{P}_q (S)$ the set of polynomials of total degree less than or equal to $q$ over the set $S$. 
We introduce the finite dimensional spaces $X_h(\Omega) \subset H^1_0(\Omega,\Gamma_D)$ and $W_h(\Gamma_C)$ such that 
\begin{eqnarray}
\begin{split}
X_h(\Omega)&:=\left\{v_{h} \in \mathscr{C}(\bar{\Omega}): \tforall \tau \in \mathcal{T}^{h},~ v_{h}|_{\tau} \in \mathcal{P}_{1}(\tau),~ v_{h} |_{\Gamma_{D}}=0\right\}, \\
W_{h}\left(\Gamma_{C}\right) &:=\left\{\psi_{h}=v_{h}|_{\Gamma_C}: v_{h} \in X_{h}(\Omega)\right\} \\ &=\left\{\psi_{h} \in \mathscr{C}\left(\bar{\Gamma}_{C}\right): \tforall t \in \mathcal{T}^{h}_{C}, ~ \psi_{h} |_t \in \mathcal{P}_{1}(t)\right\} .
\end{split}
\end{eqnarray}
We enforce strongly the non-negativity condition on $\phi=\nabla u \cdot \boldsymbol{ n}|_{\Gamma_C}$. 
We define 
\begin{equation}
\label{ }
M_h(\Gamma_C):=\{\psi_h\in W_h(\Gamma_C):\psi_h\geq 0\}.
\end{equation}
The fine-grid solution $(u_h, \phi_h)\in X_h(\Omega)\times M_h(\Gamma_C) $  satisfies
\begin{eqnarray}
\begin{split}
a(u_h,v_h)-b(\phi_h,v_h)&=L(v_h) & \quad \tforall v_h\in X_h(\Omega),\\
b(\psi_h-\phi_h, u_h)&\geq 0 & \quad \tforall \psi_h\in M_h(\Gamma_C).
\end{split}
\label{finehybridineq}
\end{eqnarray}

\noindent
{\bf Remark:} 
In order to carry out the numerical analysis of this approximation, let us define the discrete closed convex set
\begin{equation}
\label{ }
K_h(\Omega)=\{v_h\in X_h(\Omega),~   b(\psi_h,v_h)\geq 0, \tforall \psi_h \in  M_h(\Gamma_C)\}.
\end{equation}
It is a nonconforming approximation of $K(\Omega)$, i.e., $K_h(\Omega)\not\subset K(\Omega)$. We have $u_h\in K_h(\Omega)$, and from inequality in \eqref{finehybridineq} we deduce that $b(\phi_h,u_h)=0$. By some manipulations, one can derive that $u_h$ is also a solution of the Signorini problem in a fine-scale variational inequality format: Find $u_h\in K_h(\Omega)$ such that
\begin{equation}
\label{finevarineq}
a(u_h,v_h-u_h)\geq L(v_h-u_h), ~~\tforall v_h\in K_h(\Omega).
\end{equation}
Testing the inequality  \eqref{finehybridineq} with $\psi_h=0$ and $\psi_h=2\phi_h$ leads to the system
\begin{eqnarray}
\begin{split}
a(u_h,v_h)-b(\phi_h,v_h)&=L(v_h) & \quad \tforall v_h\in X_h(\Omega),\\
b(\psi_h,u_h)&\geq 0 & \quad \tforall \psi_h\in M_h(\Gamma_C),\\
b(\phi_h, u_h)&=0. \\
\end{split}
\label{eqn:fine-hy}
\end{eqnarray}
Let $\eta_i$ ($i\in \{1,\cdots, M_1\}$) be the basis for $X_h(\Omega)$ (with $M_1:=\text{dim}(X_h(\Omega))$) and $\xi_j$ ($j\in \{1,\cdots, M_2\}$) be the basis for $M_h(\Gamma_C)$ (with $M_2 := \text{dim}(M_h(\Gamma_C))$). 
We write $u_h=\sum_{i=1}^{M_1} u_i\eta_i$ and $\phi_h=\sum_{j=1}^{M_2} \phi_j \xi_j$ for some coefficients $u_i$ and $\phi_j$. In terms of matrix representations, the above problem \eqref{eqn:fine-hy} can be written as 
\begin{eqnarray}
\begin{split}
M_{\text{fine}} \vec{U}_h-B^T_{\text{fine}}\vec{ \Phi}_h&= \vec{L},\\
B_{\text{fine}}\vec{U}_h& \geq 0, \\
\vec{\Phi}_h^TB_{\text{fine}}\vec{U}_h& =0,\\
\vec{\Phi}_h&\geq0,
\end{split}
\label{FINEeq}
\end{eqnarray}
where the matrices and vectors are defined as follows: 
$$ M_{\text{fine}} := \left [ a(\eta_j, \eta_i) \right ] \in \mathbb{R}^{M_1 \times M_1}, \quad B_{\text{fine}} := \left [ b(\xi_j, \eta_i) \right ] \in \mathbb{R}^{M_2 \times M_1}, $$
$$\vec{L} := \left [ L(\eta_i) \right ] \in \mathbb{R}^{M_1}, \quad \text{and} \quad  \vec{\Phi}_h := \left [ \phi_j \right ] \in \mathbb{R}^{M_2}.$$
We denote a vector is non-negative if each component of the vector is non-negative. 

\section{Multiscale Method} \label{sec:gmsfem}
In this section, we develop the multiscale method for solving the Signorini problem in coarse grid. Our approach is based on the framework of the generalized multiscale finite element methods \cite{efendiev2013generalized}. 
We first construct the multiscale space for approximation. Then, we derive the coarse-grid equation for the Signorini problem. 

\subsection{Multiscale Space}
We present the construction of the multiscale space for coarse-grid approximation.
First of all, we start with an appropriate snapshot space that can capture the fine-scale features of the solutions. Then, we apply a 
space reduction technique to select dominant modes in the snapshot space. 
    
In this work, we simply use the local fine-grid basis functions in $X_h(\Omega)$ as the snapshot space. 
The space reduction technique is based on a carefully designed local spectral problem giving a rapidly decaying residual.  
The resulting reduced space is obtained by the linear span of these dominant modes. 

In practice, 
the dimension reduction and the construction of the multiscale basis functions are performed
locally. 
In particular, the local spectral problem is solved for each local domain $\omega_i$ whose center point is an interior coarse grid node. 
We denote $\eta_j^{(i)}$ basis function of $X_h(\Omega)$ such that $\text{supp}(\eta_j^{(i)}) \subset \omega_i$. 
We define $X_h(\omega_i)$ to be the fine-scale finite element space corresponding to the local domain $\omega_i$ such that
$$X_{h}(\omega_i):=\operatorname{span}\left\{\eta_{j}^{(i)}: 1 \leq j \leq J_{i}\right\},$$
where $J_i$ is the number of fine grid nodes in the local domain $\omega_i$. 
The local spectral problem reads as follows: Find $(v,\lambda)\in X_{h}(\omega_i) \times \mathbb{R}$ such that 
\begin{equation}
\label{spectralpb}
a_i(v,w)=\lambda s_i(v,w) \quad \tforall w\in X_{h}(\omega_i),
\end{equation}
where
$$a_i(v,w):= \int_{\omega_i}  \kappa(x)  \nabla v \cdot \nabla w ~ dx, \quad 
s_i(v,w):= \int_{\omega_i}\kappa(x)   v\cdot  w~ dx.$$
Assume that the eigenvalues of \eqref{spectralpb} are arranged in increasing order
$$ 
0 \leq \lambda_1^{(i)}<\lambda_2^{(i)}<\cdots<\lambda_{J_i}^{(i)},
$$
where $\lambda_k^{(i)}$ denotes the $k$-th eigenvalue for the local domain $\omega_i$. The corresponding eigenvectors are denoted by $Z_k^{(i)}=(Z_{kj}^{(i)})_{j=1}^{J_i}$ where $Z_{kj}^{(i)}$ is the $j$th component of the vector $Z_k^{(i)}$. We select the first $0<l_i\leq J_i$ eigenfunctions to form the multiscale space. 
For simplicity, we assume the eigenvalues are strictly increasing. In practice, if there are multiple eigenvectors corresponding to a specific eigenvalues, then we will take all eigenvectors to be part of the basis functions when the corresponding eigenvalue is in the chosen range of eigenvalues. 

Using the eigenfunctions, multiscale basis functions can be constructed as 
$$
v_k^{(i)}= \sum_{j=1}^{J_i} Z_{kj}^{(i)} \eta_j^{(i)}, \quad k=1,2,\cdots, l_i.
$$

\textbf{Additional multiscale basis for $\Gamma_C$.} Inspired by the technique used in \cite{spiridonov2019generalized} when it deals with the boundary condition, we also need additional multiscale basis to tackle the unilateral boundary condition.
In order to ensure the multiscale solution is positive on $\Gamma_C$ 
we introduce a class of additional basis functions, namely $v_{\text{ex} 1} ^{(i)}$ , which are the solutions of the following Dirichlet problems \eqref{ex1}. 
We remark that the additional basis functions are only needed in the local domains $\omega_i$ whose central vertex $\boldsymbol{a_i}$ lies on $\Gamma_C$.  In other words, $\boldsymbol{a_i}$ is in the set $ \{\boldsymbol{a_i}\}_{1 \leq i \leq N_0}$.

\begin{equation}\label{ex1}
\begin{cases}
-\nabla \cdot (\kappa(x) \nabla v_{\text{ex} 1} ^{(i)})=0, ~~x\in \omega_i\\
v_{\text{ex} 1} ^{(i)}=\delta_i(x),\quad 
x\in \partial \omega_i.
\end{cases}
\end{equation}

Here, $\delta_i(x)$ is the piecewise linear function along the boundary that takes the value $1$ at $\boldsymbol{a_i}$ and zero elsewhere (other coarse grid nodes on boundary $\partial \omega_i$).

The global offline space is then 
$$
X_{H}(\Omega):=\operatorname{span} \left \{ v_k^{(i)}: 1\leq k\leq l_i, 1\leq i \leq N \right \}\bigoplus \operatorname{span} \left \{v_{\text{ex} 1}^{(i)}\\:  1\leq i\leq N_0 \right \}.
$$
For simplicity, we use the single-index notation and we write 
$$
X_{H}(\Omega)=\operatorname{span}\{v_k : 1\leq k\leq M_{\text{off}}\},
$$
where $M_{\text{off}}:=\sum_{i=1}^{N} l_i+ N_0$ is the total number of multiscale basis functions. This space will be used for approximating the solution of the Signorini problem over the coarse grid. 

The approximation of the space of Lagrange multiplier is made by choosing a closed convex cone $M_H(\Gamma_C)$ on the coarse grid such that  

\begin{equation}
\label{ }
M_H(\Gamma_C):=\{\psi_H \in W_H(\Gamma_C):~\psi_H\geq 0\}
\end{equation}
where 
$$ W_H \left ( \Gamma_C \right ) := \left\{\psi_{H} \in \mathscr{C}\left(\bar{\Gamma}_{C}\right): \tforall t \in \mathcal{T}^{H}_{C}, ~ \psi_{H} |_t \in \mathcal{P}_{1}(t)\right\}.$$

\subsection{The Method}
 Once the multiscale spaces are constructed, the multiscale method reads as follows: find $(u_H, \phi_H)\in  X_H(\Omega) \times M_H(\Gamma_C)$ such that 
\begin{eqnarray}
\begin{split}
a(u_H, v_H)-b(\phi_H,v_H)&=L(v_H) & \quad \tforall v_H\in  {X_H(\Omega) },  \\
b(\psi_H-\phi_H, u_H)&\geq 0& \quad \tforall \psi_H\in M_H(\Gamma_C). 
\end{split}
\label{GMeq}
\end{eqnarray}
The above formulation \eqref{GMeq} is equivalent to the following system: 
\begin{eqnarray}
\begin{split}
a(u_H,v_H)-b(\phi_H,v_H)&=L(v_H) &\quad \tforall v_H\in  X_H(\Omega),\\
b(\psi_H,u_H)&\geq 0 & \quad \tforall \psi_H\in M_H(\Gamma_C),\\
b(\phi_H, u_H)&=0. 
\end{split}
\end{eqnarray}
If we write 
\begin{equation}
K_H(\Omega):=\{v_H\in  {X_H(\Omega) }: ~ b(\psi_H,v_H)\geq 0 ~ \tforall \psi_H \in  M_H(\Gamma_C)\},
\end{equation}
the corresponding variational inequality of \eqref{GMeq} is to find 
$u_H\in K_H(\Omega)$ such that
\begin{equation}
\label{GMvarineq}
a(u_H,v_H-u_H)\geq L(v_H-u_H)  \quad \tforall v_H\in K_H(\Omega).
\end{equation}

Notice that each multiscale basis function $v_k$ is eventually represented on the fine grid. 
Therefore, each $v_k$ can be represented by a vector $V_k$ containing the coefficients in the expansion of $v_k$ in the fine-grid basis functions. Then, we define
$$
R_{\text{off}}:=\left [V_1 ~\cdots ~V_{M_{\text{off}}} \right ],
$$
which maps from the multiscale space to the fine-scale space. Similar to \eqref{FINEeq}, the coarse-grid system \eqref{GMeq} can be written in matrix form as follows: 
\begin{eqnarray}
\begin{split}
R_{\text{off}}^TM_{\text{fine}} R_{\text{off}}\vec{U}_H-R_{\text{off}}^TB^T_{\text{fine}}G_H\vec{ \Phi}_H&= R_{\text{off}}^T\vec{L}, \\
G_H^TB_{\text{fine}}R_{\text{off}}\vec{U}_H& \geq 0,\\
\vec{\Phi}_H^TG_H^TB_{\text{fine}}R_{\text{off}}\vec{U}_H& =0,\\
\vec{\Phi}_H&\geq0.
\end{split}
\label{GMmeq}
\end{eqnarray}
Here, $G_H$ is a prolongation operator (see \cite{grossmann2007numerical,briggs2000multigrid}) from $M_H(\Gamma_C) $ into $M_h(\Gamma_C)$ and it is defined in such a way that the values at points on the coarse gird map unchanged to the fine grid, while the values at fine-grid points not on the coarse grid are the linear interpolation of their coarse-grid neighbors. 
Moreover, $\vec{U}_H$ and $\vec{ \Phi}_H$ are vectors of coefficients in the expansions of the solutions $u_H$ and $\phi_H$ in the space $X_H(\Omega) $ and $M_H(\Gamma_C)$, respectively.  

\subsection{Algorithmic Details}
In this section, we discuss some algorithmic details of the proposed multiscale method. 
Following \cite{ulbrich2011semismooth}, the three constraints in \eqref{GMmeq} can be written as a single nonlinear equation by using a NCP-function. Thus, the coarse-grid system \eqref{GMmeq} is equivalent to the following one: 
\begin{eqnarray}
\begin{split}
M_{\text{coarse}}\vec{U}_H-B_{\text{coarse}}^T \vec{ \Phi}_H&= R_{\text{off}}^T\vec{L}, \\
\vec{\Phi}_H-\max\{0,\vec{\Phi}_H-c(B_{\text{coarse}} \vec{U}_H)\}&=0,
\end{split}
\end{eqnarray}
with any $c>0$, $M_{\text{coarse}} :=R_{\text{off}}^TM_{\text{fine}} R_{\text{off}}$, and  $B_{\text{coarse}} :=G_H^TB_{\text{fine}}R_{\text{off}}$. 
In this work, we perform the primal-dual active set strategy \cite{ito, gustafsson2017mixed} to solve this hybrid problem.
See Algorithm \ref{primaldual} for the details. 
In the algorithm, given a matrix $C$ and a row position vector $\boldsymbol{i}$, we denote by $C(\boldsymbol{i},:)$ the submatrix consisting of the rows of $C$ marked by the index vector $\boldsymbol{i}$.  
Similarly, $b(\boldsymbol{i})$ consists of the components of vector $b$ whose indices are in vector $\boldsymbol{i}$. 
In each iteration, the linear system to be solved at Step 9 has the saddle point structure. 
\begin{algorithm}[H] 
\caption{Primal-dual active set method for the hybrid problem}
\begin{algorithmic}[1]
\State  \textbf{Input: }$c$,  $TOL$\label{primaldual}
\State $k=0$; $\vec{\Phi}_H^0=\boldsymbol{0}$
\State Solve $ M_{\text{coarse}} \vec{U}_H^{0}= R_{\text{off}}^T\vec{L} $
\While{$k<1$ or $\|\vec{\Phi}_H^k-\vec{\Phi}_H^{k-1}\|>TOL$} 
        \State $\boldsymbol{s}^k=\vec{\Phi}_H^k-c(B_{\text{coarse}}\vec{U}_H^k)$
        \State Let $\boldsymbol{i}^k$ consists of the indices of the non-positive elements of $\boldsymbol{s}^k$
        \State Let $\boldsymbol{a}^k$ consists of the indices of the positive elements of $\boldsymbol{s}^k$
        \State $\vec{\Phi}_H^{k+1}(\boldsymbol{i}^k)=\boldsymbol{0}$
        \State Solve
       $$
        \left[ \begin{array}{cc}
       M_{\text{coarse}} & -\big(B_{\text{coarse}}\left(\boldsymbol{a}^{k},:\right)\big)^T\\
        -B_{\text{coarse}}\left(\boldsymbol{a}^{k},:\right) & \mathbf{0}
        \end{array}\right]
       \left[\begin{array}{c}
       \vec{U}_H^{k+1} \\ 
       \vec{\Phi}_H^{k+1}\left(\boldsymbol{a}^{k}\right)
       \end{array}\right]
       =\left[\begin{array}{c}
        R_{\text{off}}^T\vec{L} \\ 
       \boldsymbol{0}\left(\boldsymbol{a}^{k}\right)
       \end{array}\right]$$
       \State $k=k+1$
    \EndWhile 
    \State {\bf Output:} the coarse-grid solution $\vec{U}_k^H$.     
 \end{algorithmic}
\end{algorithm}

In practice, instead of checking the accumulating difference term with a given tolerance, we set a maximal number of iterations $\texttt{maxiter}$ such that the number of primal-dual active set iterations may not exceed $\texttt{maxiter}$.

\section{Convergence Analysis} \label{sec:analysis}
In this section, we present a convergence analysis of the proposed multiscale method for the Signorini problem.
\begin{lemma}[ {Error estimate between fine-scale and GMsFEM solutions}] 
Let $u_h\in K_h(\Omega)$ be the fine-grid solution obtained in \eqref{finevarineq} and $u_H\in K_H(\Omega)$ be the  {GMsFEM} solution obtained in \eqref{GMvarineq}. The following estimates holds:
\begin{equation}
\label{ }
\begin{aligned}
\|u_h-u_H\|_{H^1(\Omega)}^2 &\leq C[ 
\inf_{v_H\in K_H(\Omega)} 
(\|u_h-v_H\|_{H^1(\Omega)}^2+\langle \phi_h, v_H\rangle_{\frac{1}{2},\Gamma_C})
+\\
&
\inf_{v_h\in K_h(\Omega)} (\langle \phi_h, v_h-u_H  \rangle_{\frac{1}{2},\Gamma_C})
]\\
\end{aligned}
\end{equation}
\end{lemma}
We start with a suitable inf-sup condition for the multiscale finite element space.
To this aim, we assume that the coarse partition $\mathcal{T}^C_H$ is quasi-uniform. 
The $L^2_\kappa(\Gamma_C)$-orthogonal projection $\pi_H: W_h(\Gamma_C)\to W_H(\Gamma_C)$ is defined by 
$$\int_{\Gamma_C}\kappa (\pi_H u_h)\cdot v_H~ d\Gamma=\int_{\Gamma_C} \kappa u_h\cdot v_H ~~d\Gamma \tforall u_h\in W_h(\Gamma_C), v_H\in W_H(\Gamma_C).$$
Since the partition of the domain is quasi-uniform and the permeability field $\kappa$ is piecewise constant, from \cite{bramble1991some} we know that $\pi_H$ has the following properties:

\begin{enumerate}
    \item For any $\mu\in[0,1]$, $\pi_H$ is stable  in the sense that for any $\psi\in W_h(\Gamma_C)$ 
\begin{equation}
\|\pi_H\psi\|_{H^\mu_\kappa(\Gamma_C)}\leq C\|\psi\|_{H^\mu_\kappa(\Gamma_C)}.
\end{equation}
    \item The following approximation results hold: for any $\psi\in W_h(\Gamma_C)$,
\begin{equation}\label{eq3.6}
\|\psi-\pi_H\psi\|_{H^{-\frac12}_\kappa(\Gamma_C)}\leq  H\|\psi\|_{H^{\frac12}_\kappa(\Gamma_C)} \text{ and } 
\|\psi-\pi_H\psi\|_{L^{2}_\kappa(\Gamma_C)}\leq C H\|\psi\|_{H^{1}_\kappa(\Gamma_C).}
\end{equation}
\end{enumerate}
The result for the fractional norm is by Riesz-Thorin Theorem, while the boundedness of the error with respect to $H^{-\frac12}_\kappa(\Gamma_C)$ norm is obtained by duality.
For the inf-sup condition in the fine-scale space, it has been proven in \cite{bend}. We remark that this condition still holds in weighted norm setting due to the properties of the $L^2_\kappa(\Gamma_C)-$orthogonal projection. 
Assume that the triangulation
$\mathcal{T}_h^C$ is quasi-uniform at the fine-scale level, the inf-sup condition in the fine-scale space holds:
\begin{equation}\label{infsupfine}
\inf_{\psi_h\in W_h(\Gamma_C)} \sup_{v_h\in X_h(\Omega)} \frac{b(\psi_h,v_h)}{\|v_h\|_{H^1_\kappa(\Omega)}\|\psi_h\|_{H^{-\frac{1}{2}}_\kappa(\Gamma_C)}}\geq \gamma_0
\end{equation}
where the constant $\gamma_0$ does not depend on $h$.

As a key component in the proof of the convergence result, we first prove the inf-sup condition in the multiscale space. 
\begin{lemma}
Under the assumption that the mesh $\mathcal{T}^C_H$ is quasi-uniform, the following inf-sup condition holds:
\begin{equation}\label{infsupMs}
\inf_{\psi_H\in W_H(\Gamma_C)} \sup_{v_H\in X_H(\Omega)} \frac{b(\psi_H,v_H)}{\|v_H\|_{H^1_\kappa(\Omega)}\|\psi_H\|_{H^{-\frac{1}{2}}_\kappa(\Gamma_C)}}\geq \gamma. 
\end{equation}
The constant $\gamma>0$ does not depend on $H$.
\end{lemma}
\begin{proof}
Let $\psi_H\in W_H(\Gamma_C).$ We are going to construct $v_H\in X_H(\Omega)$ satisfying
\begin{equation}\label{eq3.7H}
b(\psi_H,v_H)\geq\|\psi_H\|_{H^{-\frac{1}{2}}_\kappa(\Gamma_C)}^2 \quad \text{and} \quad \gamma\|v_H\|_{H^1_\kappa(\Omega)}\leq \|\psi_H\|_{H^{-\frac{1}{2}}_\kappa(\Gamma_C)}
\end{equation} 
for some constant $\gamma$.
Let us consider the function $v_h\in X_h(\Omega)$ such that
\begin{equation}
\int_{\Omega}\kappa(x) v_hw_h ~dx+\int_{\Omega}\kappa(x) \nabla v_h\cdot\nabla w_h ~dx= b(\psi_H, w_h)  \quad \tforall w_h\in X_h(\Omega). 
\end{equation}
The following stability inequalities hold: for some constants $c_{-}$ and $c_+$, 
\begin{equation}\label{eq3.9H}
c_-\|\psi_H\|_{H^{-\frac{1}{2}}_\kappa(\Gamma_C)}\leq \|v_h\|_{H^1_\kappa(\Omega)}\leq c_+\|\psi_H\|_{H^{-\frac{1}{2}}_\kappa(\Gamma_C)}, 
\end{equation}
where the first inequality implies from the inf-sup condition in the fine-scale finite element space.
Then we set $v_H\in X_H(\Omega)$ such that $v_H|_{\Gamma_C}=\pi_H(v_h|_{\Gamma_C})$ and
\begin{equation}\label{eq3.10H}
\|v_H\|_{H^1_\kappa(\Omega)}\leq c\|\pi_H v_h\|_{H^{\frac{1}{2}}_\kappa(\Gamma_C)}\leq c\|v_h\|_{H^{\frac{1}{2}}_\kappa(\Gamma_C)}.
\end{equation}
This function $v_H$ is built using the stable finite element extension operator similar to the one mentioned in \cite{bernardi1998local}. We will discuss the construction of such an extension operator on the multiscale spaces in Appendix \ref{sec:StableExtension}.

The first statement of \eqref{eq3.7H} is valid because
\begin{equation}\label{eq3.11}
b(\psi_H,v_H)=b(\psi_H,v_h)=\|v_h\|^2_{H^1_\kappa(\Omega)}\geq c_-^2\|\psi_H\|^2_{H^{-\frac{1}{2}}_\kappa(\Gamma_C)}.
\end{equation}
The second is obtained from \eqref{eq3.9H} and \eqref{eq3.10H} together with the trace theorem.
\end{proof}

The derivation of the error estimate between $u_h$ and $u_H$ is based on an adaptation of Falk's lemma. Throughout this section, we denote $a \lesssim b$ if there exists a generic constant $C>0$ independent of the domain and the heterogeneous coefficient such that $a \leq C b$. 
\begin{lemma}
[{\bf Error estimate between fine- and coarse-scale solutions}]
\label{hlemma4.1Ms}
Let $(u_h,\phi_h)$ be the fine-scale solution of the hybrid Signorini problem \eqref{finehybridineq} and $(u_H,\phi_H)\in X_H(\Omega)\times M_H(\Gamma_C)$ the coarse-scale solution of the hybrid Signorini problem \eqref{GMeq}. 
The following estimates holds:
\begin{equation}
\label{lem4-2}
\begin{aligned}
\norm{u_h-u_H}_{H^1_\kappa(\Omega)}^2 & \lesssim 
\inf_{v_H\in K_H(\Omega)} 
\left (\norm{u_h-v_H}_{H^1_\kappa(\Omega)}^2+b(\phi_h, v_H) \right )
+
\inf_{v_h\in K_h(\Omega)} b(\phi_h, v_h-u_H).
\end{aligned}
\end{equation}
\end{lemma}
\begin{proof}
From the equivalence between the variational inequality formulation and the hybrid formulation, we know that $u_h$ is also in $K_h(\Omega)$ satisfying  \eqref{finevarineq} and $u_H$ is also in $K_H(\Omega)$ satisfying \eqref{GMvarineq}.
From \eqref{finevarineq} and \eqref{GMvarineq}, we have that
$$
\begin{aligned}
\label{}
   a(u_h,u_h-v_h) &\leq L(u_h-v_h) &\tforall v_h\in K_h(\Omega),   \\
    a(u_H,u_H-v_H) &\leq L(u_H-v_H) &\tforall v_H\in K_H(\Omega).   
\end{aligned}
$$
Adding these inequalities and transposing terms, we obtain
$$
a(u_h,u_h)+a(u_H,u_H)\leq L(u_h-v_h)+L(u_H-v_H)+a(u_h,v_h)+a(u_H,v_H).
$$
Subtracting $a(u_h,u_H)+a(u_H,u_h)$ from both sides and grouping terms, we obtain
$$
\begin{aligned}
a(u_h-u_H,u_h-u_H)
&\leq L(u_h-v_H)+L(u_H-v_h)-a(u_h,u_H-v_h)-a(u_H, u_h-v_H)\\
&= \left [L(u_h-v_H)-a(u_h, u_h-v_H) \right ]+ \left [L(u_H-v_h)-a(u_h,u_H-v_h) \right ]\\
&\quad +a(u_h-u_H,u_h-v_H)\\
&=b(\phi_h, v_H-u_h)+b(\phi_h, v_h-u_H)+a(u_h-u_H,u_h-v_H)\\
&=b(\phi_h, v_H) +b(\phi_h,v_h-u_H)+a(u_h-u_H,u_h-v_H).
\end{aligned}
$$
In the last step we have used $b(\phi_h, u_h)=0$.
Using the continuity and coercivity of the bilinear form $a(\cdot,\cdot)$ and the Cauchy-Schwarz inequality, we have 
$$\begin{aligned}
\norm{u_h-u_H}_{H^1_\kappa(\Omega)}^2&\lesssim \norm{u_h-v_H}_{H^1_\kappa(\Omega)}^2+b(\phi_h, v_H)
+b(\phi_h, v_h-u_H).
\end{aligned}
$$
This completes the proof. 
\end{proof}

\noindent {\bf Remark:} The first term in the right hand-side of \eqref{lem4-2} is the approximation error, while the second term is the consistency error. 

The following theorem shows that the distance between the element $v_h\in X_h(\Omega)$ and the multiscale finite element space $X_H(\Omega) $ we construct is bounded by the norm of $v_h$ up to a specific constant. The readers are referred to \cite{li2019convergence} for more details of the approximibility of the multiscale space. 

Recall that $(\lambda^{(i)}_j, v^{(i)}_j)$ is the eigenpair obtained from \eqref{spectralpb}. Based on the eigenfunctions, we define the interpolation operator $I_H^{(i)} :  X_h (\omega_i) \to X_H(\omega_i)$ by 
\begin{equation}
\displaystyle I_H^{(i)} v_h:=\sum_{k=1}^{{l_i}} s(v_h, v_k^{(i)}) v_k^{(i)}
\label{eqn:interpolation}
\end{equation}
for any $v_h \in X_h(\omega_i)$. 
\begin{thm}\label{thm1.1}
Assume that $\omega_i$ is a local domain with central vertex $\boldsymbol{a_i}$. 
Then, the following estimate holds
\begin{equation}
\inf_{r_H\in X_H(\omega_i) } \norm{v_h-r_H}_{H^1_\kappa(\omega_i)}^2 \lesssim \left (1+\frac{1}{\Lambda}\right)|v_h|_{H^1_\kappa(\omega_i)}^2
\end{equation}
for any $v_h \in X_h(\omega_i)$, where $\displaystyle\Lambda:=\min_{1\leq i\leq N} \lambda_{l_i+1}^{(i)}$. 
\end{thm}
\begin{proof}
The set of the eigenfunctions $\left \{v_1^{(i)},\cdots ,v_{J_i} ^{(i)} \right \}$ forms an orthonormal basis of the fine-scale finite element space $X_h(\omega_i)$ with respect to $s(\cdot, \cdot)$ inner product. 
On the other hand, $X_H(\omega_i)$ is a finite-dimensional space making up with the first $l_i$ eigenfunctions $\left \{v_1^{(i)},\cdots ,v_{l_i} ^{(i)} \right \}$ for the spectral problem \eqref{spectralpb} with corresponding eigenvalues $\lambda_1^{(i)}<\lambda_2^{(i)}<\lambda_3^{(i)}<\cdots<\lambda_{l_i}^{(i)}$. Given any $v_h\in X_h(\omega_i), $  we can write 
$$v_h=\sum_{k=1}^{J_i} s(v_h, v_k^{(i)})v_k^{(i)}.$$
Then, we have 
$$|v_h|_{H^1_\kappa(\omega_i)}^2 = \sum_{k=1}^{J_i} \lambda_k^{(i)}\abs{s(v_h,  v_k^{(i)})}^2 \quad \text{and} \quad 
\norm{v_h}_{L^2_\kappa(\omega_i)}^2=\sum_{k=1}^{J_i} \abs{s(v_h, v_k^{(i)})}^2$$
due to the orthogonality of the eigenfunctions. 
We estimate the interpolation error as follows: 
$$
|v_h- I_H^{(i)} v_h|^2_{H^1_\kappa(\omega_i)}
=\left|\sum_{k={l_i}+1}^{J_i} s(v_h, v_k^{(i)}) v_k^{(i)}\right|^2_{H^1_\kappa(\omega_i)}
\leq |v_h|_{H^1_\kappa(\omega_i)}^2.
$$
On the other hand, we have 
$$
\begin{aligned}
\norm{v_h-I_H^{(i)} v_h}^2_{L^2_\kappa(\omega_i)}
\leq\frac{1}{\lambda_{{l_i}+1}^{(i)}}|v_h- I_H^{(i)} v_h|^2_{H^1_\kappa(\omega_i)}
& \leq  \frac{1}{\Lambda}|v_h|_{H^1_\kappa(\omega_i)}^2.
\end{aligned}
$$
Combining the estimates for the interpolation error yields the desired estimate. 
\end{proof}

\begin{lemma}\label{hlemma4.5Ms}
Let $(u_h,\phi_h)$ be the fine-scale solution of the hybrid Signorini problem \eqref{finehybridineq} and $(u_H,\phi_H)\in X_H(\Omega)\times M_H(\Gamma_C)$ the coarse-scale solution of the hybrid Signorini problem \eqref{GMeq}. 
Then, the following estimate holds: 
$$\begin{aligned}
&~~\inf _{v_H\in K_{H}(\Omega)}\left(\left\|u_h-v_{H}\right\|_{H^{1}_\kappa(\Omega)}^{2}+b(\phi_h,v_{H})\right)\\
&\lesssim H\|u_h\|_{H^1_\kappa(\Omega)}^2
+\left (1+\frac{1}{\Lambda} \right )|u_h|_{H^{1}_\kappa(\Omega)}^2
+H \left\|\phi_h\right\|_{H^{\frac{1}{2}}_\kappa\left(\Gamma_{C}\right)}
\|u_h\|_{H^{1}_\kappa(\Omega)},
\end{aligned}
$$
\end{lemma}

\begin{proof}
Based on the local interpolation operator define in \eqref{eqn:interpolation}, we define the global interpolation operator $I_H$ such that $\displaystyle I_H u_h:=\sum_{1\leq i\leq N} I_H^{(i)} u_h $ for any $u_h \in X_h(\Omega)$. Then, the following error estimate holds
\begin{equation}\label{eq3.1}
\begin{aligned}
\|u_h-I_H  u_h\|_{H^{1}_\kappa(\Omega)}^2&\lesssim \sum_{1\leq i\leq N}\|u_h-I_H  u_h\|_{H^{1}_\kappa(\omega_i)}^2\\
&\leq\sum_{1\leq i\leq N}\left (1+\frac{1}{\Lambda}\right) |u_h|_{H^1_\kappa(\omega_i)}^2 
\lesssim  \left (1+\frac{1}{\Lambda} \right )|u_h|_{H^1_\kappa(\Omega)}^2.
\end{aligned}
\end{equation}
Next, we choose $v_H=I_H u_h+w_H,$ with $w_H\in X_H(\Omega)$ a stable extension of the trace function 
$\left(\pi_{H}\left({u_h}_{\mid \Gamma_{C}}\right)-\left(I_H u_h\right)_{\mid \Gamma_{C}}\right)$. 
See Appendix \ref{sec:StableExtension} for the construction of the stable extension. 
By definition, one has ${v_H}_{\mid\Gamma_C}=\pi_{H}\left({u_h}_{\mid \Gamma_{C}}\right)$. Since ${u_h}\in K_h(\Omega)$, we have
$$b\left(\psi_{H}, v_{H}\right)=\int_{\Gamma_{C}} \kappa\pi_{H}\left({u_h}_{\mid \Gamma_{C}}\right) \psi_{H} d \Gamma=\int_{\Gamma_{C}}\kappa\left({u_h}_{\mid \Gamma_{C}}\right) \psi_{H} d \Gamma \geq 0$$
for any $\psi_H\in M_H(\Gamma_C)$. 
Furthermore, using \eqref{eq3.1} and \eqref{eq3.6}, we have 
$$\begin{aligned}
\left\|u_h-v_{H}\right\|_{H^{1}_\kappa(\Omega)}&\leq\left\|u_h-I_H u_h\right\|_{H^{1}_\kappa(\Omega)}+\left\|w_H\right\|_{H^{1}_\kappa(\Omega)} \\
&\lesssim \|u_h-I_H u_h\|_{H^1_\kappa(\Omega)} +\|\pi_H ({u_h}_{\mid\Gamma_C})-(I_H u_h)_{\mid\Gamma_C}\|_{H^{\frac{1}{2}}_\kappa\left(\Gamma_{C}\right)}\\
&\lesssim \left (1+\frac{1}{\Lambda} \right )^{\frac{1}{2}}|u_h|_{H^1_\kappa(\Omega)}+\|\pi_H({u_h}_{\mid\Gamma_C})-{u_h}_{\mid\Gamma_C}\|_{H^{\frac{1}{2}}_\kappa\left(\Gamma_{C}\right)}
+\|(u_h-I_H {u_h})_{\mid\Gamma_C}\|_{H^{\frac{1}{2}}_\kappa\left(\Gamma_{C}\right)}\\
&\lesssim \left (1+\frac{1}{\Lambda}\right)^{\frac{1}{2}}|u_h|_{H^1_\kappa(\Omega)}+
H^{\frac{1}{2}}\|u_h\|_{H^1_\kappa\left(\Gamma_C\right)} 
+\left(1+\frac{1}{\Lambda}\right)^{\frac{1}{2}}|u_h|_{H^{1}_\kappa(\Omega)}\\
&\lesssim H^{\frac{1}{2}}\sum_{e\in \mathcal{T}^h_C}\|u_h\|_{H^1_\kappa(e)}+\left(1+\frac{1}{\Lambda}\right)^{\frac{1}{2}}|u_h|_{H^{1}_\kappa(\Omega)}\\
& \lesssim\sum_{e\in \mathcal{T}^h_C} H^{\frac{1}{2}}\|u_h\|_{H^2_\kappa(\tau_e)} +\left (1+\frac{1}{\Lambda}\right)^{\frac{1}{2}}|u_h|_{H^1_\kappa(\Omega)}\\
& \lesssim H^{\frac{1}{2}}\|u_h\|_{H^1_\kappa(\Omega)} +\left (1+\frac{1}{\Lambda}\right)^{\frac{1}{2}}|u_h|_{H^1_\kappa(\Omega)}.
\end{aligned}$$
Here, $\tau_{e}$ is the fine-grid element having the edge $e$. 
In order to evaluate the integral term $b(\phi_h, v_H)$, we use the saturation condition $$\int_{\Gamma_C}\kappa u_h\phi_h d \Gamma=0$$ 
and again the estimate \eqref{eq3.6}. Thus, we obtain 
$$\begin{aligned}
\int_{\Gamma_{C}}\kappa \phi_hv_{H} d \Gamma
&=\int_{\Gamma_{C}}\kappa\phi_h\left(\pi_{H} u_h-u_h\right) d \Gamma \leq \left\|\phi_h\right\|_{H^{\frac{1}{2}}_\kappa\left(\Gamma_{C}\right)}\left\|\pi_{H} u_h-u_h\right\|_{H^{-\frac{1}{2}}_\kappa\left(\Gamma_{C}\right)}\\
&\lesssim H \left\|\phi_h\right\|_{H^{\frac{1}{2}}_\kappa\left(\Gamma_{C}\right)}\|u_h\|_{H^{1}_\kappa(\Omega)}.
\end{aligned}$$
This completes the proof. 
\end{proof}
\begin{lemma}\label{hlemma4.6Ms}
Let $(u_h,\phi_h)$ be the fine-scale solution of the hybrid Signorini problem \eqref{finehybridineq} and $(u_H,\phi_H)\in X_H(\Omega)\times M_H(\Gamma_C)$ be the coarse-scale solution of the hybrid Signorini problem \eqref{GMeq}. 
Then, the consistency error is bounded as follows:
$$\inf _{v_h \in K_h(\Omega)} b(\phi_h,v_h -u_{H}) \lesssim  H^{\frac{1}{2}}\|\phi_h\|_{H^{\frac12}_\kappa(\Gamma_C)} 
\left (\left\|u_h-u_{H}\right\|_{H^{1}_\kappa(\Omega)}+ h |u_h|_{H^1_\kappa(\Omega)} \right ).
$$
\end{lemma}
\begin{proof}
Choosing $v_h=u_h$, we have for all $\psi_H\in M_H(\Gamma_C)$,
$$
b(\phi_h,v_h -u_{H}) =\int_{\Gamma_{C}}\kappa\left(\phi_h-\psi_{H}\right)\left(u_h-u_{H}\right) d \Gamma+\int_{\Gamma_{C}} \kappa\psi_{H}\left(u_h-u_{H}\right) d \Gamma.
$$
Taking $\psi_H=r_H(\phi_h)$ (where $r_H$ is a non-standard Cl\'ement type interpolation operator preserving non-negativity \cite{clement}), the first integral can be estimated as follows: 
\begin{equation}\label{eq4.4}
\begin{aligned}
\int_{\Gamma_{C}}\kappa\left(\phi_h- \psi_{H}\right)\left(u_h-u_{H}\right) d \Gamma
& \leq \|\phi_h- \psi_{H}\|_{L^2_\kappa(\Gamma_C)}\|u_h-u_H\|_{L^2_\kappa(\Gamma_C)}\\
&\lesssim H^{\frac{1}{2}}\|\phi_h\|_{H^{\frac12}_\kappa(\Gamma_C)}\left\|u_h-u_{H}\right\|_{H^{1}_\kappa(\Omega)}.
\end{aligned}
\end{equation}
Working out the remaining term requires some preliminary technical steps. Taking into account of $b(\psi_H, u_H)\geq 0$ 
and the saturation condition, 
we obtain
\begin{equation}\label{eq4.5}
\begin{aligned} \int_{\Gamma_{C}} \kappa \psi_{H}\left(u_h-u_{H}\right) d \Gamma & \leq \int_{\Gamma_{C}} \kappa\psi_{H} u_h d \Gamma=\int_{\Gamma_{C}}\kappa\left(\psi_{H}-\phi_h\right) u_h d \Gamma \\ & \leq \sum_{t \in \mathcal{T}^{H}_{C}} \int_{t}\kappa\left(\psi_{H}-\phi_h\right) u_h d \Gamma\\
& \leq \sum_{t \in \mathcal{T}^{H}_{C}}\left\|\psi_{H}-\phi_h\right\|_{L^{2}_\kappa(t)}\|u_h\|_{L^{2}_\kappa(t)} 
\end{aligned}
\end{equation}
For each fine-scale segment $e=e_{i,h}=( x_{i,h}^C,  x_{i+1,h}^C)$, we define $T_e$ as the union of $e$ and its neighbors, for example, $T_{e_{1,h}}={e_{1,h}}\cup{e_{2,h}}=( x_{1,h}^C,  x_{3,h}^C)$, $T_{e_{2,h}}={e_{1,h}}\cup{e_{2,h}}\cup{e_{3,h}}=( x_{1,h}^C,  x_{4,h}^C)$.
Then, focusing on each segment $e$, if ${u_h}_{\mid T_e}>0$ then $\phi_h=0$, this implies that ${\psi_H}_{\mid e}=0$; otherwise, $u_h$ has at least one zero in $T_e$. This results in, for any $x\in T_e$, from the weighted Poincar\'{e} inequality for 1D case (see \cite{pechstein2013weighted}), we have $$\|u_h\|_{L^2_\kappa(e)}\leq C h_e|u_h|_{H^1_\kappa(T_e)}.$$
Hence, 
$$\displaystyle\begin{aligned}
\sum_{t\in \mathcal{T}^H_C}\|u_h\|_{L^2_\kappa(t)}^2&=\sum_{t\in \mathcal{T}^H_C}\sum_{e\subseteq t} \|u_h\|_{L^2_\kappa(e)}^2\leq c h^2\sum_{t\in \mathcal{T}^H_C}\sum_{e\subseteq t} |u_h|_{H^1_\kappa(T_e)}^2\\
&=c h^2\sum_{t\in \mathcal{T}^H_C}\sum_{e\subseteq t} \sum_{\tilde{e}\subseteq T_e}|u_h|_{H^1_\kappa(\tilde{e})}^2
\leq c h^2\sum_{t\in \mathcal{T}^H_C}\sum_{e\subseteq t} \sum_{\tilde{e}\subseteq T_e}\|\nabla u_h\|_{L^2_\kappa({\tilde{e}})}^2\\
&\lesssim h^2\sum_{t\in \mathcal{T}^H_C}\sum_{e\subseteq t} \sum_{\tilde{e}\subseteq T_e}\|\nabla u_h\|_{H^1_\kappa(\tau_{\tilde{e}})}^2= h^2\sum_{t\in \mathcal{T}^H_C}\sum_{e\subseteq t} \sum_{\tilde{e}\subseteq T_e}|u_h|_{H^1_\kappa(\tau_{\tilde{e}})}^2\\
&\lesssim h^2|u_h|^2_{H^1_\kappa(\Omega)},
\end{aligned}$$
where $\tau_{\tilde{e}}$ is the fine-grid element having the edge $\tilde{e}$. 
Using \eqref{eq4.5} and the Cauchy-Schwartz inequality, we obtain 
\begin{equation}\label{eq4.6}
\begin{aligned}
\int_{\Gamma_{C}} \kappa\psi_{H}\left(u_h-u_{H}\right) d\Gamma &\lesssim \left (\sum_{t \in \mathcal{T}^{H}_{C}} \left\|\psi_{H}-\phi_h\right\|_{L^2_\kappa(t)}^2 \right )^{\frac{1}{2}} h|u_h|_{H^{1}_\kappa\left(T_{t}\right)}\\
& \lesssim H^{\frac{1}{2}} \|\phi_h\|_{H^{\frac12}_\kappa(\Gamma_C)}h|u_h|_{H^1_\kappa(\Omega)}.
\end{aligned}
\end{equation}
This completes the proof. 
\end{proof}

\noindent {\bf Remark:} In the proof above, we have used the approximation property of non-standard Cl\'{e}ment type interpolation operator. 
This approximation property still holds under the weighted norm due to the weighted Poincar\'{e} inequality with the given assumptions on the weight function $\kappa$. 

Using Lemma \ref{hlemma4.1Ms}, together with the bounds for the approximation and the consistency errors, we are now in the position to state the main result of this work, namely, the error bound of the proposed multiscale method. 
\begin{thm}\label{hthm4.3Ms}
Let $(u_h,\phi_h)$ be the fine-scale solution of the hybrid Signorini problem \eqref{finehybridineq} and $(u_H,\phi_H)\in X_H(\Omega)\times M_H(\Gamma_C)$ the coarse-scale solution of the hybrid Signorini problem \eqref{GMeq}.  
Then, the following error estimate holds: 
\begin{equation}\label{eq4.3}
\begin{aligned}
\|u_h-u_H\|_{H^1_\kappa(\Omega)} 
\lesssim  H ^{\frac{1}{2}} \left ( \|\phi_h\|_{H^{\frac{1}{2}}_\kappa(\Gamma_C)} +\|\kappa^{-\frac{1}{2}}f\|_{L^2(\Omega)} \right )+
\left ( 1+\frac{1}{\Lambda}\right )^{\frac{1}{2}}  \|\kappa^{-\frac{1}{2}}f\|_{L^2(\Omega)}.
\end{aligned}
\end{equation}

\end{thm}
\begin{proof}
Using Lemma \ref{hlemma4.1Ms} with the results of Lemmas \ref{hlemma4.5Ms} and \ref{hlemma4.6Ms} yields the inequality
$$
\begin{aligned}
\|u_h-u_H\|_{H^1_\kappa(\Omega)}^2
&\lesssim H\|u_h\|_{H^1_\kappa(\Omega)}^2+\left (1+\frac{1}{\Lambda} \right )|u_h|_{H^1_\kappa(\Omega)}^2+H\left\|\phi_h\right\|_{H^{\frac12}_\kappa(\Gamma_C)} \|u_h\|_{H^1_\kappa(\Omega)}\\
&+H^{\frac{1}{2}}\|\phi_h\|_{H^{\frac12}_\kappa(\Gamma_C)}\left\|u_h-u_{H}\right\|_{H^{1}_\kappa(\Omega)}+ H^{\frac{1}{2}} \|\phi_h\|_{H^{\frac12}_\kappa(\Gamma_C)}h|u_h|_{H^1_\kappa(\Omega)} 
\end{aligned}
$$
from which, combining the fact that $\|u_h\|_{H^1_\kappa(\Omega)}\leq \|\kappa^{-\frac{1}{2}} f\|_{L^2(\Omega)}$, we derive the final estimate \eqref{eq4.3}: 
$$
\begin{aligned}
&~~~\left(\|u_h-u_H\|_{H^1_\kappa(\Omega)}-\frac{C}{2}H^{\frac12}\|\phi_h\|_{H^{\frac12}_\kappa(\Gamma_C)}\right)^2\\
&\leq C \left ( H\|u_h\|_{H^1_\kappa(\Omega)}^2+\left(1+\frac{1}{\Lambda}\right)|u_h|_{H^1_\kappa(\Omega)}^2 \right )
+CH\left(\frac{\|\phi_h\|_{H^{\frac12}_\kappa(\Gamma_C)}^2+\|u_h\|_{H^1_\kappa(\Omega)}^2}{2}\right)\\
&~~~+C\left(\frac{H\|\phi_h\|_{H^{\frac12}_\kappa(\Gamma_C)}^2+h^2|u_h|_{H^1_\kappa(\Omega)}^2}{2}\right)+ \frac{C^2}{4}H\|\phi_h\|_{H^{\frac12}_\kappa(\Gamma_C)}^2\\
& \lesssim  H \|\phi_h\|^2_{H^{\frac{1}{2}}_\kappa(\Gamma_C)}+ \left (1+\frac{1}{\Lambda}\right )|u_h|_{{H^{1}_\kappa(\Omega)}}^2+H \|u_h\|^2_{H^1_\kappa(\Omega)}.
\end{aligned}
$$
This implies that 
$$
\|u_h-u_H\|_{H^1_\kappa(\Omega)}
\lesssim H ^{\frac{1}{2}} \left ( \|\phi_h\|_{H^{\frac{1}{2}}_\kappa(\Gamma_C)}+ \|\kappa^{-\frac{1}{2}}f\|_{L^2(\Omega)}\right ) + \left (1+\frac{1}{\Lambda} \right )^{\frac{1}{2}}\|\kappa^{-\frac{1}{2}}f\|_{L^2(\Omega)}.
$$
This completes the proof. 
\end{proof}

Combining the convergence rate of $u_H$ and the variational equality \eqref{GMeq}, we get the convergence rate of $\phi_H$ as a corollary.
\begin{cor}\label{cor4.7Ms}
Let $(u_h,\phi_h)$ be the fine-scale solution of the hybrid Signorini problem \eqref{finehybridineq} and $(u_H,\phi_H)\in X_H(\Omega)\times M_H(\Gamma_C)$ be the coarse-scale solution of the hybrid Signorini problem \eqref{GMeq}. 
Then, the following estimate holds 
\begin{equation}
\begin{aligned}
\|\phi_h-\phi_H\|_{H^{-\frac{1}{2}}_\kappa(\Gamma_C)}&
\lesssim  H ^{\frac{1}{2}} \left ( \|\phi_h\|_{H^{\frac{1}{2}}_\kappa(\Gamma_C)} +\|\kappa^{-\frac{1}{2}}f\|_{L^2(\Omega)} \right )+
\left ( 1+\frac{1}{\Lambda}\right )^{\frac{1}{2}}  \|\kappa^{-\frac{1}{2}}f\|_{L^2(\Omega)}.
\end{aligned}
\end{equation}
\end{cor}
\begin{proof}
Let $\psi_H\in M_H(\Gamma_C)$ and $v_H\in X_H(\Omega)$. Then
$$
b\left(\phi_{H}-\psi_{H}, v_{H}\right)=a\left(u_{H}, v_{H}\right)-L\left(v_{H}\right)-b\left(\psi_{H}, v_{H}\right)=a\left(u_{H}-u_h, v_{H}\right)+b\left(\phi_h-\psi_{H}, v_{H}\right).
$$
On account of the \textit{inf-sup} condition and the triangle inequality, we have 
$$
\begin{aligned}
\|\phi_h-\phi_H\|_{H^{-\frac12}_\kappa(\Gamma_C)}&
\leq \|\phi_h-\psi_H\|_{H^{-\frac12}_\kappa(\Gamma_C)} +\|\psi_H-\phi_H\|_{H^{-\frac12}_\kappa(\Gamma_C)}\\
&\leq \frac{1}{\gamma}\sup_{v_H\in X_H(\Omega)} \frac{b(\phi_H-\psi_H,v_H)}{\|v_H\|_{H^1_\kappa(\Omega)}}+ \|\phi_h-\psi_H\|_{H^{-\frac12}_\kappa(\Gamma_C)}\\
&\leq\frac{1}{\gamma}\sup_{v_H\in X_H(\Omega)} \frac{a\left(u_{H}-u_h, v_{H}\right)+b\left(\phi_h-\psi_{H}, v_{H}\right)}{\|v_H\|_{H^1_\kappa(\Omega)}}+ \|\phi_h-\psi_H\|_{H^{-\frac12}_\kappa(\Gamma_C)}\\
&\leq\frac{1}{\gamma}\sup_{v_H\in X_H(\Omega)} \frac{\|u_{H}-u_h\|_{H^1_\kappa(\Omega)} \|v_{H}\|_{H^1_\kappa(\Omega)}+\|\phi_h-\psi_{H}\|_{H^{-\frac12}_\kappa(\Gamma_C)}\|v_{H}\|_{H^{\frac12}_\kappa(\Gamma_C)}}{\|v_H\|_{H^1_\kappa(\Omega)}}\\
&~~~+ \|\phi_h-\psi_H\|_{H^{-\frac12}_\kappa(\Gamma_C)}.
\end{aligned}$$
Since $\psi_H \in M_H(\Gamma_C)$ is arbitrary, we have 
$$\|\phi_h-\phi_H\|_{H^{-\frac12}_\kappa(\Gamma_C)}
\lesssim \|u_h-u_H\|_{H^1_\kappa(\Omega)}+\inf_{\psi_H\in M_H(\Gamma_C)}\|\phi_h-\psi_H\|_{H^{-\frac12}_\kappa(\Gamma_C)}.$$
The estimate is completed by taking $\psi_H=r_H\phi_h\in M_H(\Gamma_C)$ (where $r_H$ is a non-standard Cl\'ement type interpolation operator preserving non-negativity \cite{clement}) and using Theorem \ref{hthm4.3Ms}. As a result, we obtain 
$$
\begin{aligned}
\|\phi_h-\phi_H\|_{H^{-\frac12}_\kappa(\Gamma_C)}
& \lesssim \|u_h-u_H\|_{H^1_\kappa(\Omega)}+\inf_{\psi_H\in M_H(\Gamma_C)}\|\phi_h-\psi_H\|_{H^{-\frac12}_\kappa(\Gamma_C)}\\
&\lesssim \|u_h-u_H\|_{H^1_\kappa(\Omega)}+\|\phi_h-r_H\phi_h\|_{L^2_\kappa(\Gamma_C)}\\
&\lesssim \|u_h-u_H\|_{H^1_\kappa(\Omega)}+H^{\frac12}\|\phi_h\|_{H^{\frac12}_\kappa(\Gamma_C)}\\
&\lesssim H ^{\frac{1}{2}} \left ( \|\phi_h\|_{H^{\frac{1}{2}}_\kappa(\Gamma_C)} +\|\kappa^{-\frac{1}{2}}f\|_{L^2(\Omega)} \right )+
\left ( 1+\frac{1}{\Lambda}\right )^{\frac{1}{2}}  \|\kappa^{-\frac{1}{2}}f\|_{L^2(\Omega)}.
\end{aligned}
$$
This completes the proof. 
\end{proof}

\section{Numerical Results} \label{sec:numerics}

In this section, we present some numerical examples to demonstrate the effectiveness and efficiency of the proposed multiscale method. We set the computational domain and boundary to be $\Omega = (0,1)^2$,  $\Gamma_D = [0,1] \times \{ 1 \}$, $\Gamma_C = [0,1] \times \{ 0 \}$, and $\Gamma_N = \partial \Omega \setminus (\Gamma_D \cup \Gamma_C)$. We partition the domain into $16 \times 16$ square elements with mesh size $H = \sqrt{2}/16$ to form the coarse grid $\mathcal{T}^H$. For each coarse block, we partition it into $16 \times 16$ square elements so that the mesh size $h$ of the fine grid $\mathcal{T}^h$ is $h = \sqrt{2}/256$. We remark that we use bilinear element on this fine grid to compute the reference solution $u_h$. In the following, we compute the relative $L^2$ and energy errors defined as follows 
$$ e_{L^2} := \frac{\norm{u_h - u_{\text{ms}}}_{L^2(\Omega)}}{\norm{u_h}_{L^2(\Omega)}} \quad \text{and} \quad e_{a} := \frac{\norm{u_h - u_{\text{ms}}}_{H_{\kappa}^1(\Omega)}}{\norm{u_h}_{H_{\kappa}^1(\Omega)}}$$
to measure the performance of the multiscale method. We remark that when performing the primal-dual active set iterations, we set $c = 1$ and the maximal number of iterations is set to be $12$. 

\begin{example} \label{exp1}
In the first example, we consider a highly heterogeneous permeability field $\kappa$ in the domain as shown in Figure \ref{fig:kappa-exp1} (left). The background value is $1$ (i.e., the blue region) and the contrast value in the channels and inclusions is $10^4$ (i.e., the yellow region). The permeability field is a piecewise constant function on the fine-grid elements. The source function $f$ is set to be $f(x_1, x_2) = \sin(2\pi x_1) \sin(2 \pi x_2)$ for any $(x_1, x_2) \in \Omega$. 
\end{example}

\begin{figure}[htbp]
    \centering
    \mbox{
    \includegraphics[width = 2.8in]{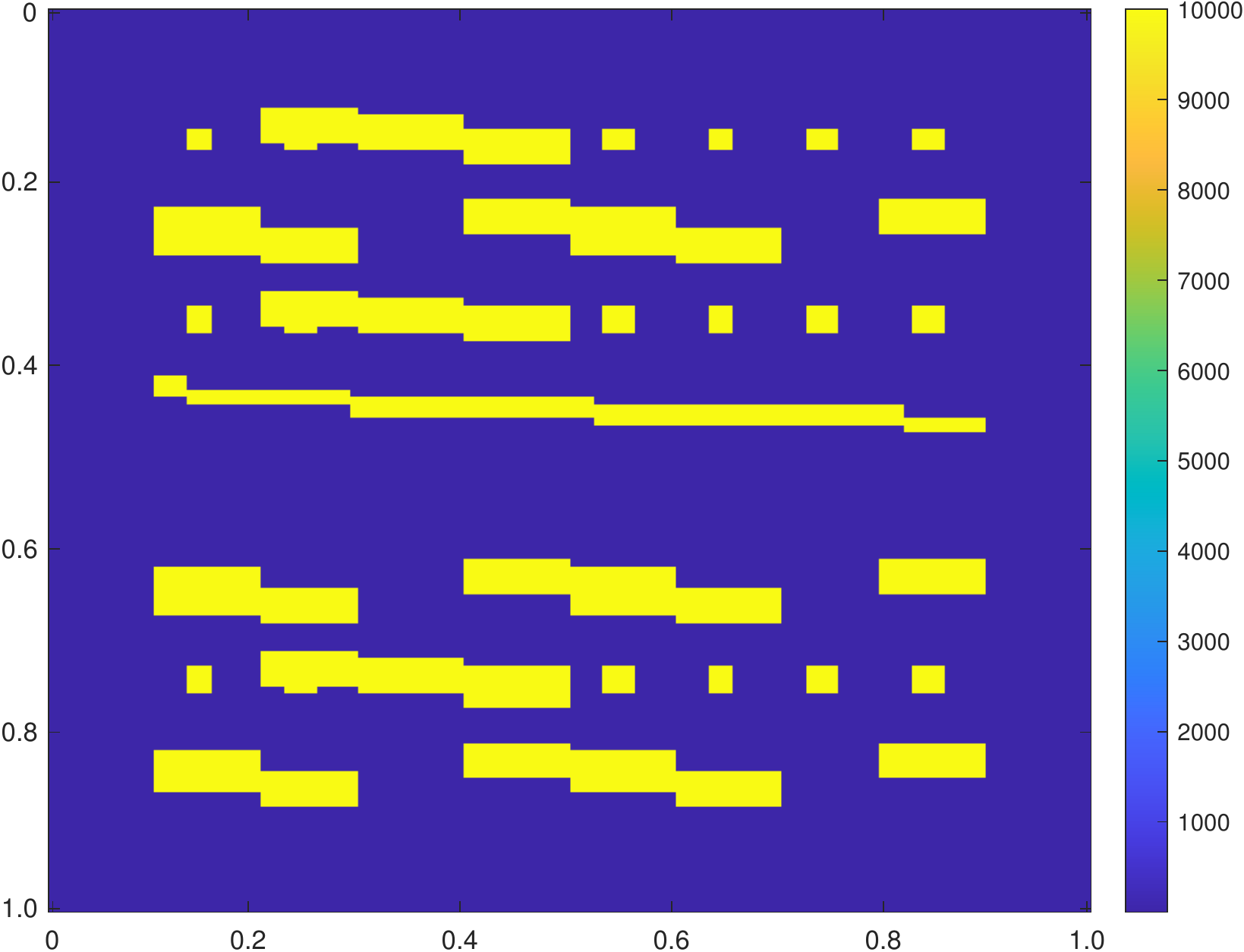}
    \includegraphics[width = 2.8in]{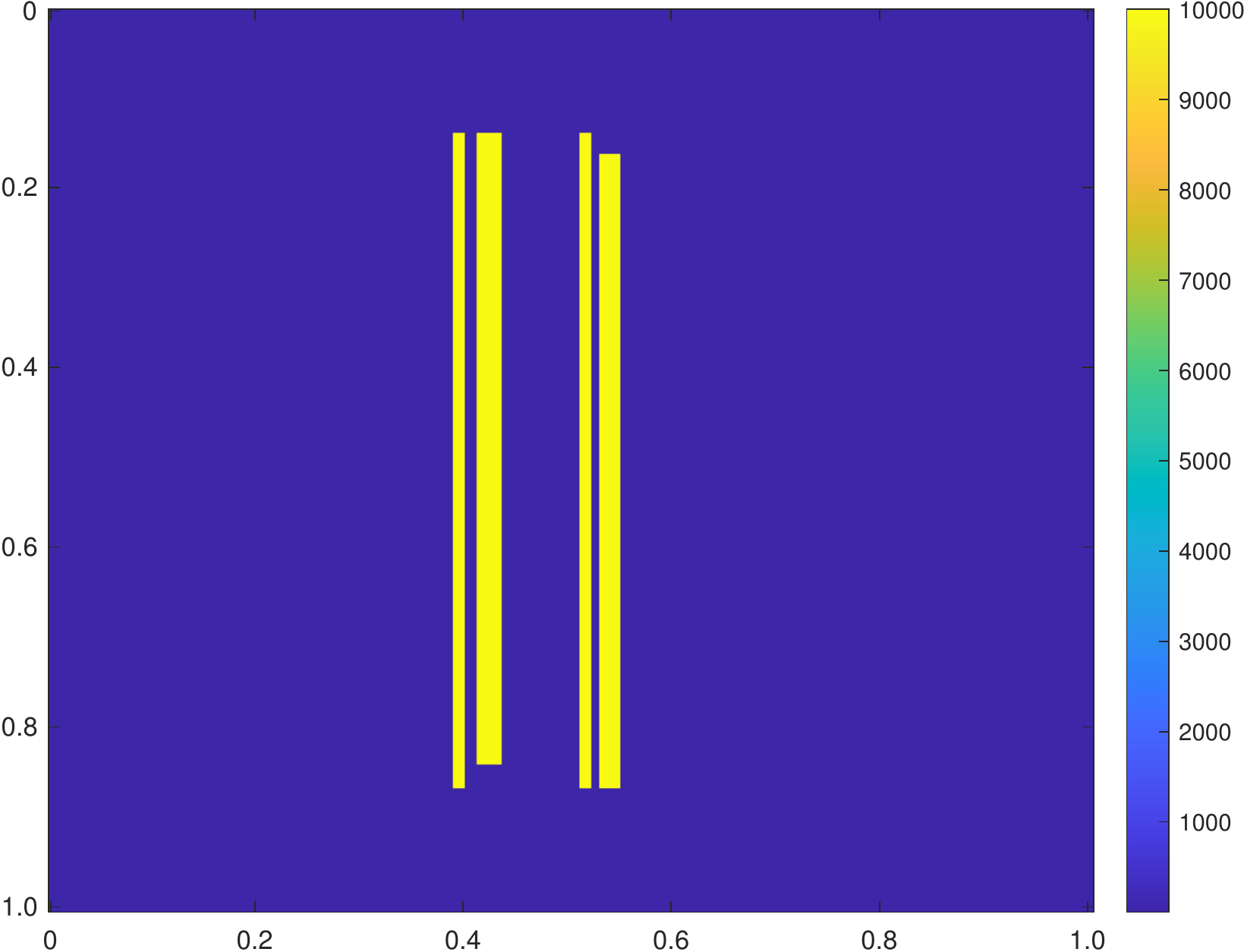}
    }
    \caption{The permeability field $\kappa$ in Example \ref{exp1} (left) and Example \ref{exp2} (right).}
    \label{fig:kappa-exp1}
\end{figure}

We present the numerical results using different numbers of local basis functions to form the multiscale space. In Table \ref{tab:error-exp1}, we present the $L^2$ and energy errors with different setting of numbers of local basis functions. The {\it coarse DOF} represents the dimension of the multiscale space. We also present the errors against the number of local basis functions in Figure \ref{fig:error-plot-exp1}. The solution profiles of the reference and multiscale solutions are plotted in Figure \ref{fig:sol-pro-exp1}. We remark that the larger number of local basis functions is, the larger the parameter $\Lambda$ is. From the results below, one can observe the convergence in terms of number of local basis functions. When five local basis functions are used, the relative $L^2$ error is around the level of $1\%$ while the relative energy error is approximately $10\%$. 

\begin{table}[htbp]
    \centering
    \begin{tabular}{c|ccccc}
    \hline \hline 
    Number of local bases & 1 & 2 & 3 & 4 & 5 \\
    \hline 
    Coarse DOF & $289$ & $514$ & $739$ & $964$ & $1189$ \\ 
    \hline 
    Energy error $e_{a}$ & $87.6568\%$ & $26.3642\%$ & $18.8316\%$ & $14.1339\%$ & $10.0682\%$ \\    
    \hline 
    $L^2$ error $e_{L^2}$ & $78.8863\%$ & $7.5842\%$ & $4.2161\%$ & $2.4489\%$ & $1.3033\%$ \\
    \hline \hline 
    \end{tabular}
    \caption{Numerical results in Example \ref{exp1}.}
    \label{tab:error-exp1}
\end{table}

\begin{figure}[htbp]
    \mbox{
    \includegraphics[width = 2.8in]{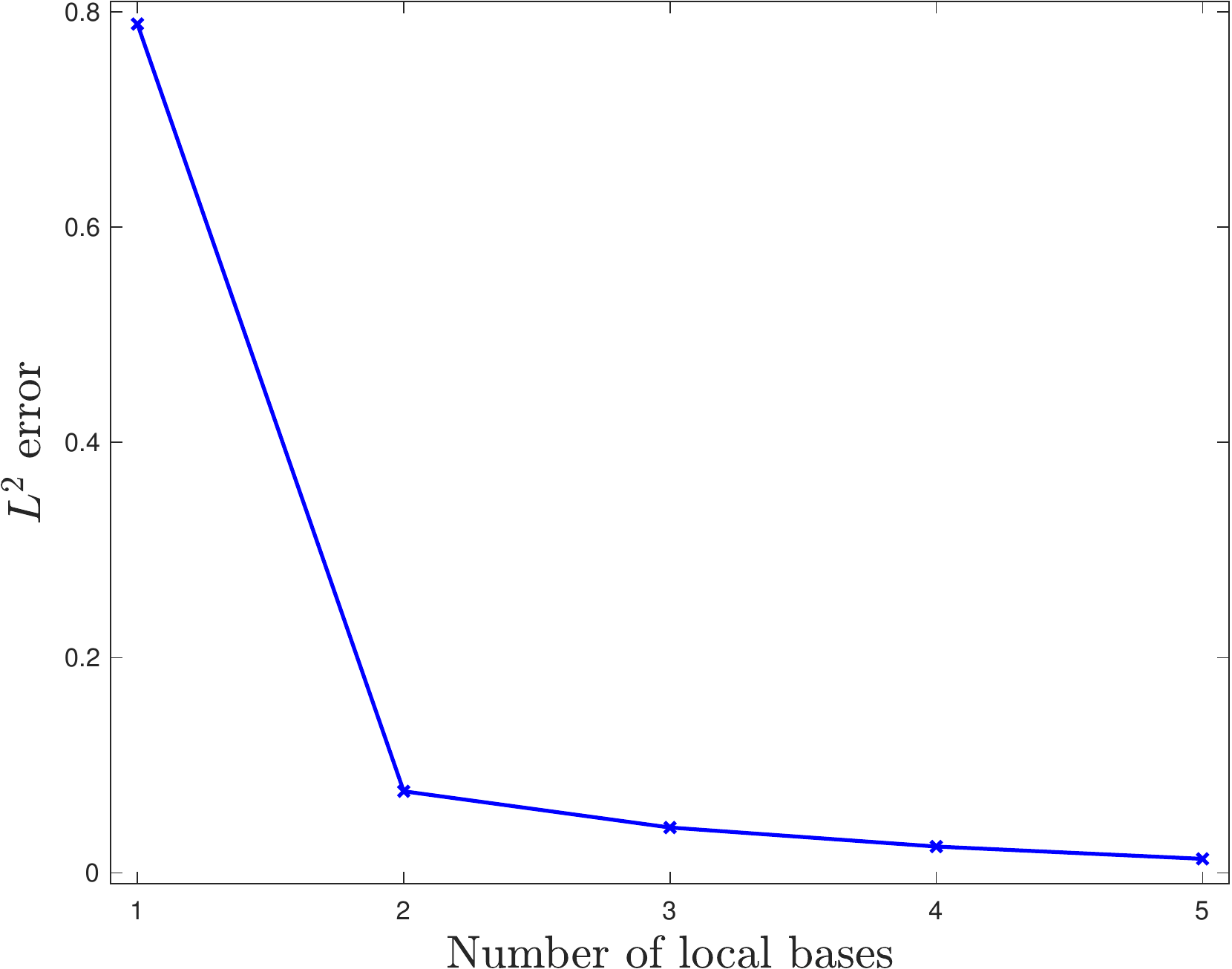}
    \includegraphics[width = 2.8in]{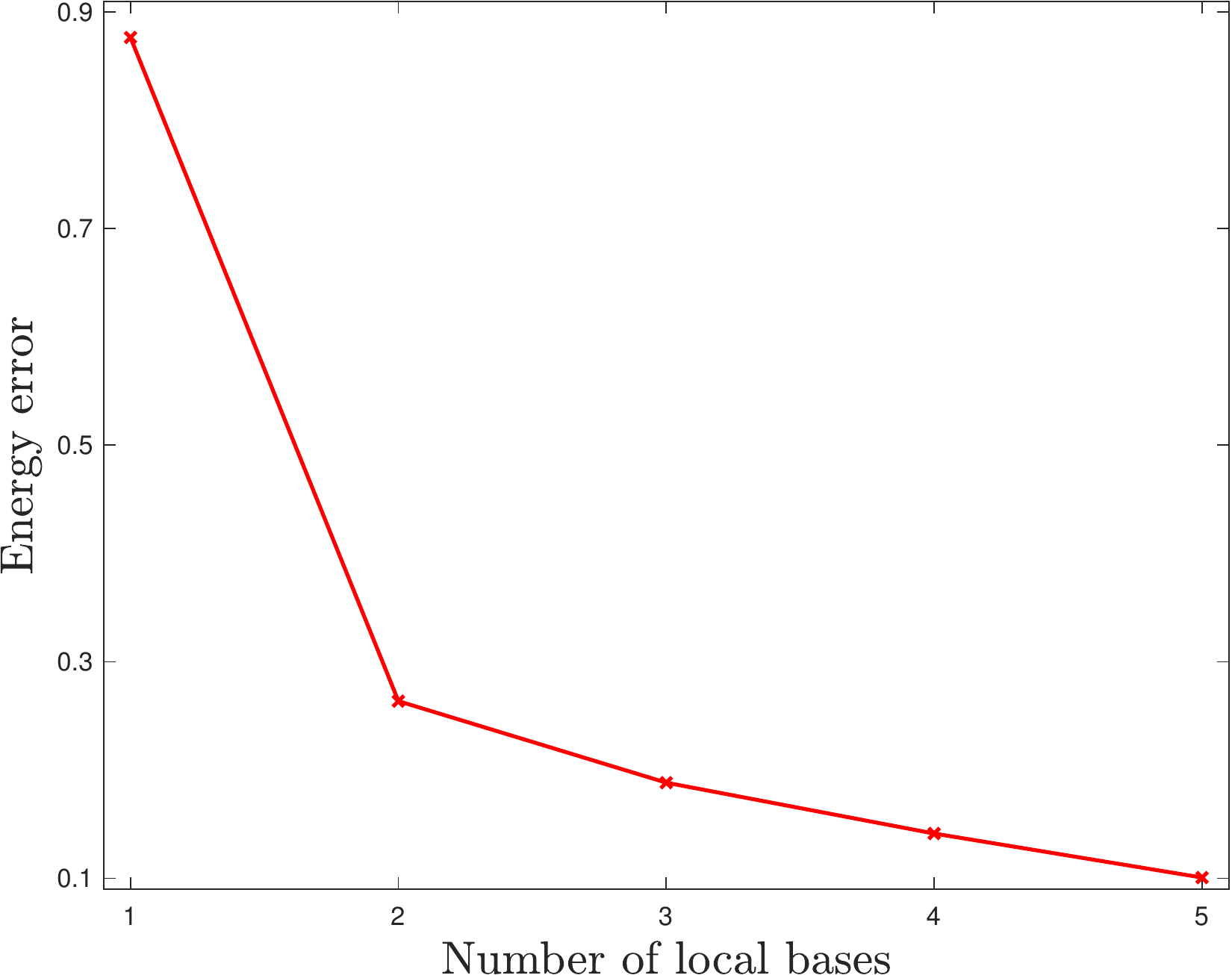}
    }
    \caption{The $L^2$ and energy errors against the number of local bases (Example \ref{exp1}).}
    \label{fig:error-plot-exp1}
\end{figure}

\begin{figure}[htbp]
    \mbox{
    \includegraphics[width = 2.8in]{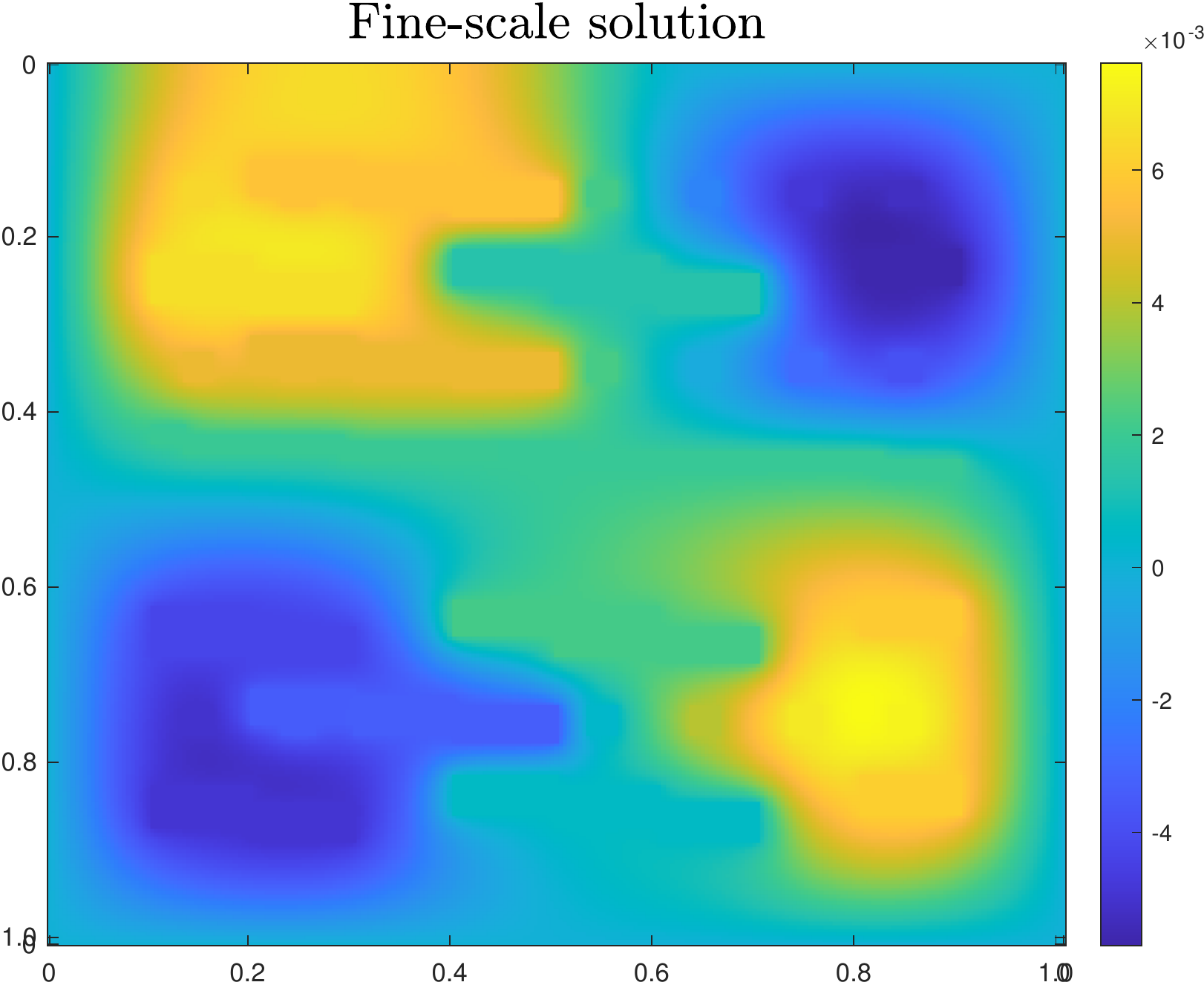}
    \includegraphics[width = 2.8in]{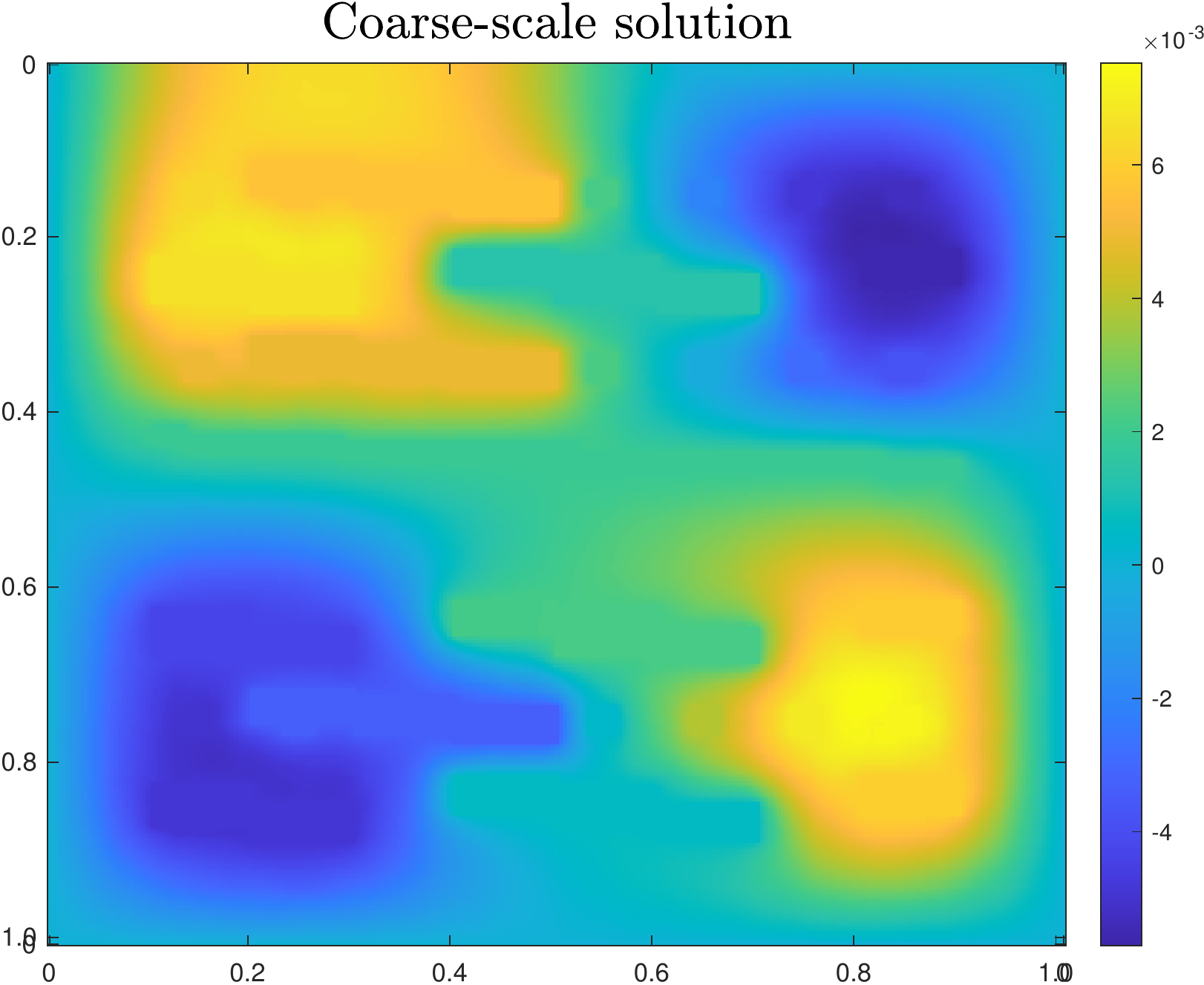}
    }
    \caption{The solution profile of Example \ref{exp1}. Left: $u_h$; right: $u_{\text{ms}}$.}
    \label{fig:sol-pro-exp1}
\end{figure}

\begin{example}
\label{exp2} 
In the second example, we consider a channelized permeability field $\kappa$ depicted in Figure \ref{fig:kappa-exp1} (right). In the background region (i.e., the blue region), the value of the channelized field is $1$. On the other hand, the value of permeability field in the channelized region (i.e., the yellow region) is $10^4$. The source function $f$ is the same as in Example \ref{exp1}. 
\end{example}

In Table \ref{tab:error-exp2}, we present the $L^2$ and energy errors with different setting of numbers of local basis functions. 
We present the errors against the number of local basis functions in Figure \ref{fig:error-plot-exp2}. The solution profiles of the reference and multiscale solutions are plotted in Figure \ref{fig:sol-pro-exp2}. 
From the results below, one can observe the convergence in terms of number of local basis functions in this example. 
In this case, one has to include (at least) two basis functions in the local multiscale space to ensure good approximation property of the multiscale method. 
When five local basis functions are used, the relative $L^2$ error is around the level of $0.5\%$ while the relative energy error is approximately $6\%$. Overall, these numerical results demonstrate that the proposed multiscale method provides a coarse-grid model with certain level of accuracy for the Signorini  problem.

\begin{table}[htbp]
    \centering
    \begin{tabular}{c|ccccc}
    \hline \hline 
    Number of local bases & 1 & 2 & 3 & 4 & 5 \\
    \hline 
    Coarse DOF & $289$ & $514$ & $739$ & $964$ & $1189$ \\ 
    \hline 
    Energy error $e_{a}$ & $27.5114\%$ & $26.3899\%$ & $6.9410\%$ & $6.7407\%$ & $6.1470\%$ \\
    \hline 
    $L^2$ error $e_{L^2}$ & $8.4982\%$ & $8.0200\%$ & $0.6928\%$ & $0.6597\%$ & $0.5693\%$ \\
    \hline \hline 
    \end{tabular}
    \caption{Numerical results in Example \ref{exp2}.}
    \label{tab:error-exp2}
\end{table}

\begin{figure}[htbp]
    \mbox{
    \includegraphics[width = 2.8in]{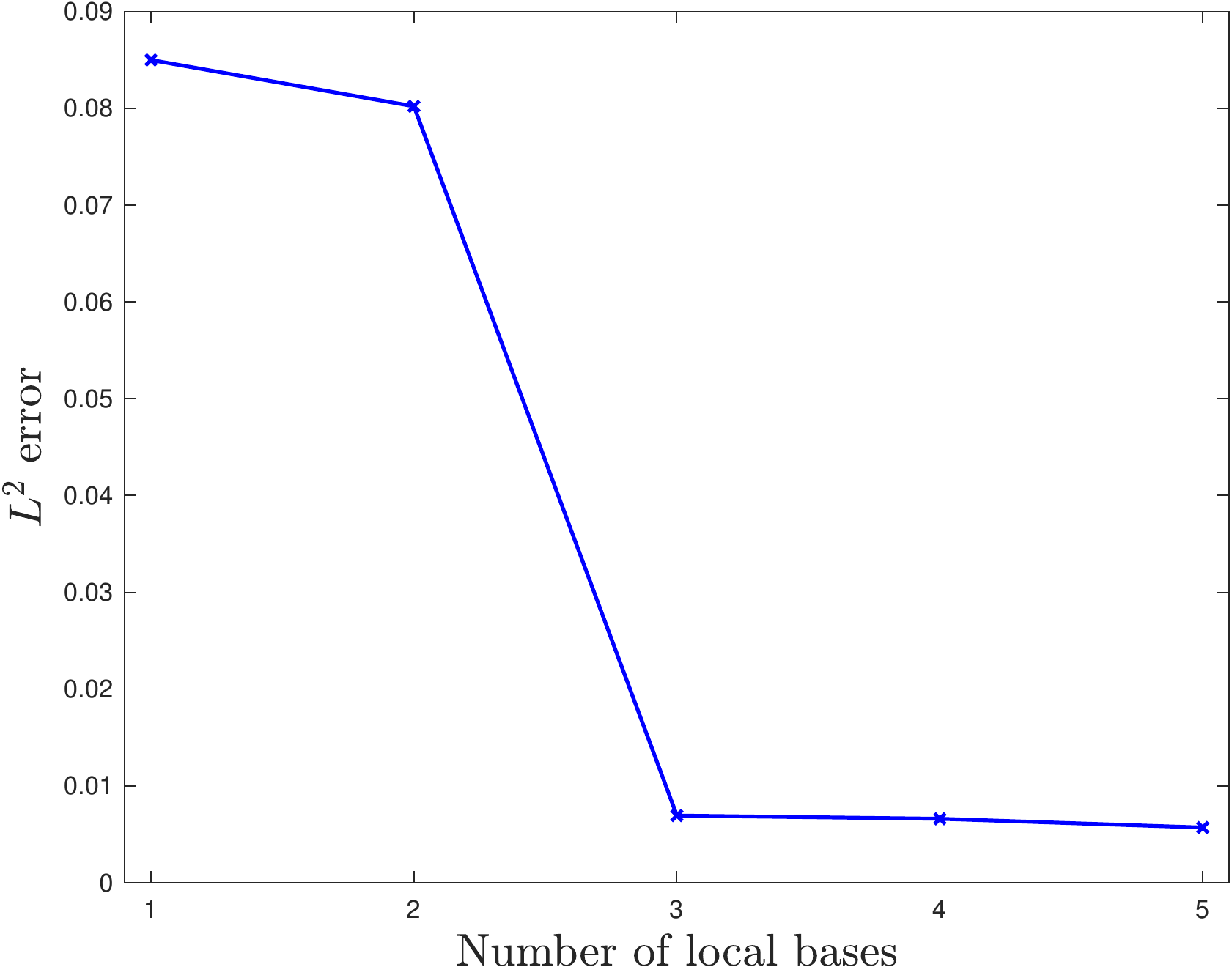}
    \includegraphics[width = 2.8in]{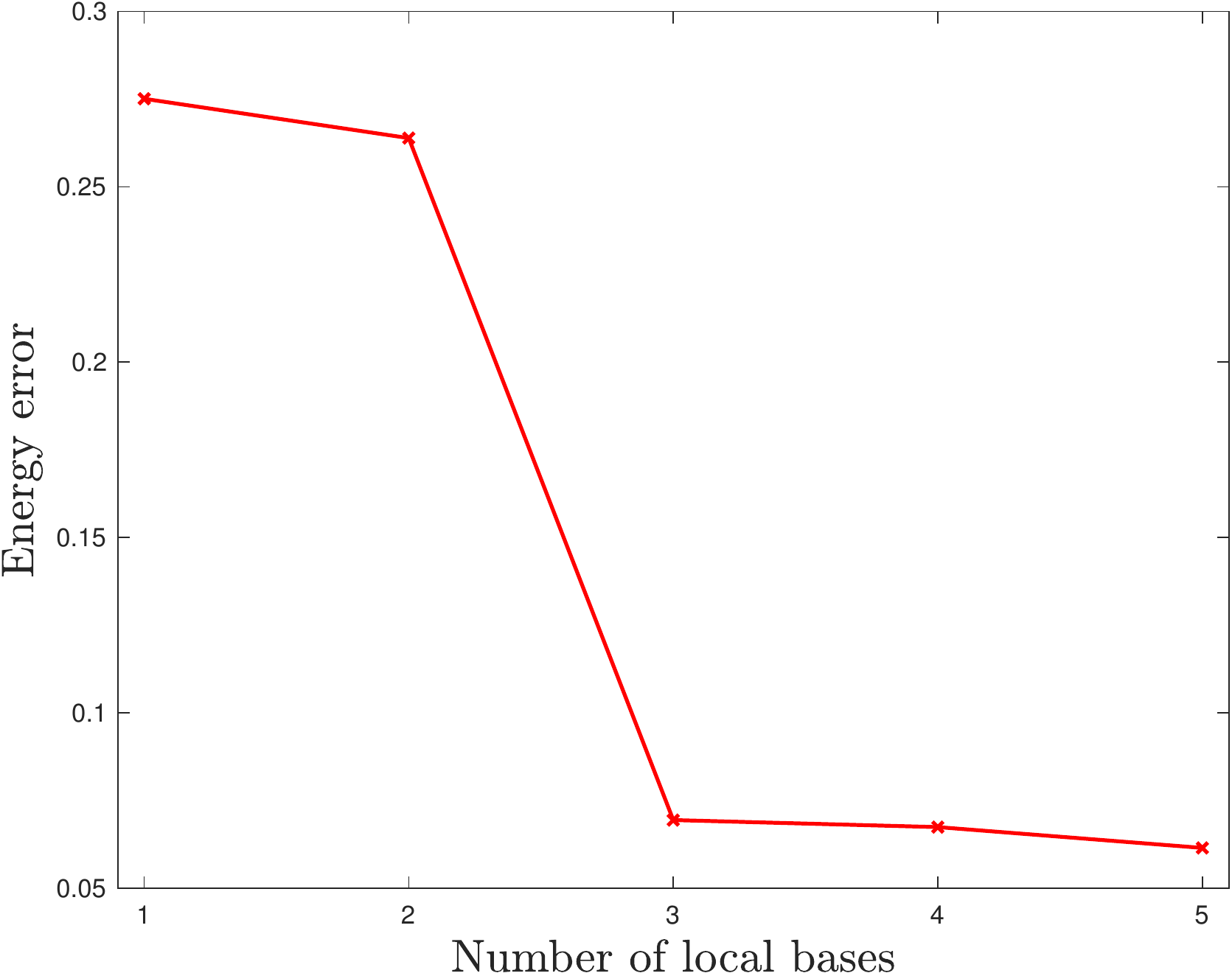}
    }
    \caption{The $L^2$ and energy errors against the number of local bases (Example \ref{exp2}).}
    \label{fig:error-plot-exp2}
\end{figure}

\begin{figure}[htbp]
    \mbox{
    \includegraphics[width = 2.8in]{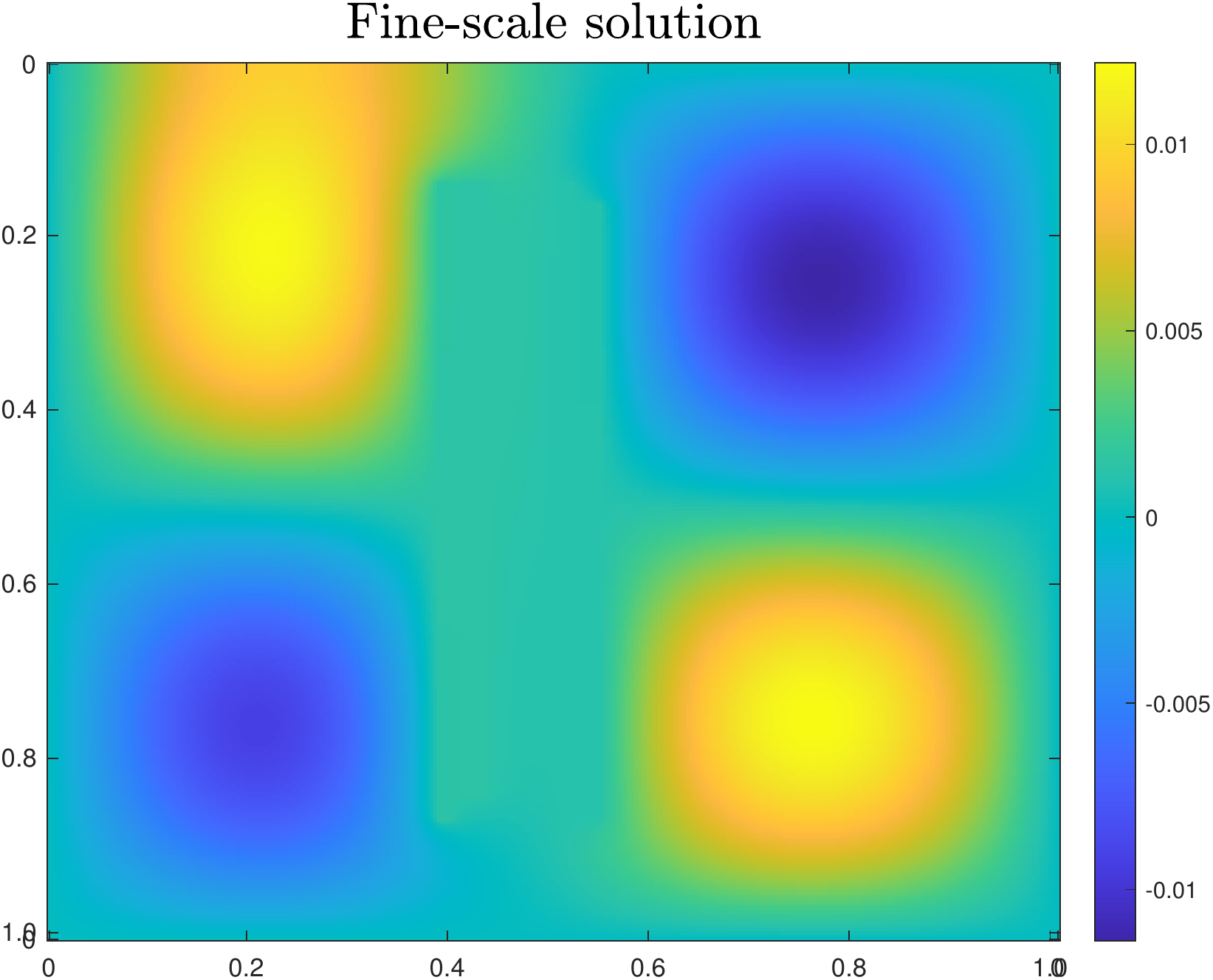}
    \includegraphics[width = 2.8in]{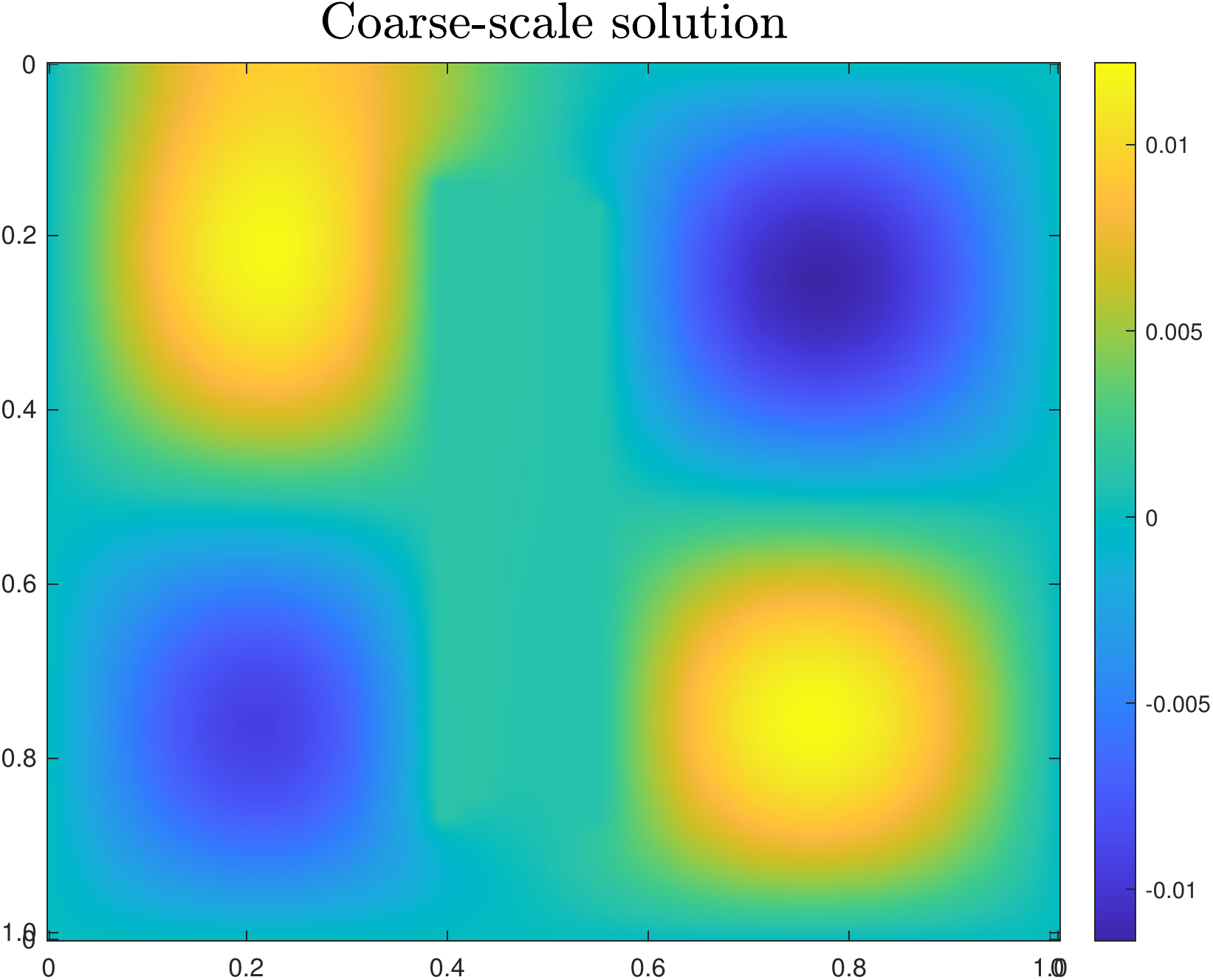}
    }
    \caption{The solution profile of Example \ref{exp2}. Left: $u_h$; right: $u_{\text{ms}}$.}
    \label{fig:sol-pro-exp2}
\end{figure}

\section{Conclusion} \label{sec:conclusion}
In this work, we developed a multiscale method for the hybrid Signorini problem with a heterogeneous permeability field based on the framework of GMsFEM. 
The hybrid formulation allows naturally the use of unilateral conditions. The construction of the multiscale basis functions is based on a space reduction method using a local spectral problem. 
We also add a set of additional basis functions to the multiscale space to ensure the unilateral boundary condition.  
This set of additional basis is obtained by solving local Dirichlet problems along the contact boundary. We provided a complete error analysis for the method. 
Several numerical results were presented to validate the analytical spectral estimate. 

\bibliographystyle{unsrt}
\bibliography{SignoriniPb}

\begin{appendices}
\section{Trace and inverse trace theorems in weighted norm}\label{sec:trace}
For $x=(x_1,\cdots, x_n)\in \mathbb{R}^n$ we denote its shift  $x^\prime:=(x_1,\cdots,x_{n-1}) \in \mathbb{R}^{n-1}$. 
Then, we define for $v\in C^\infty_c(\mathbb{R}^n)$ its trace onto the hyperplane $\mathbb{R}^{n-1}\times\{0\}$ by 
$\gamma_0 v(x^\prime):= v(x^\prime, 0)$ with $x^\prime \in \mathbb{R}^{n-1}$. 
Recall that the Fourier transform $\mathcal{F}$ is defined by
$$
\hat{v}(\xi):= \mathcal{F} v (\xi) := \int_{\mathbb{R}^n }\exp^{-i 2\pi\xi\cdot x} v(x) dx~~~ (\xi\in \mathbf{\mathbb{R}}^n)
$$

We use the Fourier transform to redefine the norms in weighted Sobolev space $H^s_\kappa(\mathbb{R}^n)$ :
$$
\|v\|_{H^s_\kappa(\mathbb{R}^n)}=\int_{\mathbb{R}^n}\kappa({\xi})\left(1+|\xi|^{2}\right)^{s}\left|\hat{v}\left(\xi\right)\right|^{2} d \xi
$$
where we extend $\kappa(x)$ by $0$ outside the domain $\Omega$.

For a bounded Lipschitz domain $\Omega$, we define $\|v\|_{H^s_\kappa(\Omega)}:= \|v_0\|_{H^s_\kappa(\mathbb{R}^n)}$ where $v_
0$ denotes the extension of $v$ by $0$ onto $\mathbb{R}^n\setminus \Omega$.


In the weighted trace theorem below, we use the weighted Sobolev norm defined through Fourier transform. We explain later that it is equivalent to the usual weighted Sobolev norm.

\begin{thm}Assume that a weight function $\kappa$ defined on the upper half space of $\mathbb{R}^n$ such that the following condition holds for some positive constant $C_0$:  $$\frac{\kappa({\xi}^\prime, 0)}{\kappa({\xi})}\leq C_0 \quad \text{for any} ~ \xi \in \mathbb{R}^{n} ~ \text{with} ~ \kappa(\xi) \neq 0.$$
Then, for $s>\frac{1}{2}$, there exists a unique extension of the trace operator $\gamma_0$ to a bounded linear operator
$$
\gamma_0 : H^s_\kappa(\mathbb{R}^n)\to H^{s-\frac{1}{2}}_\kappa(\mathbb{R}^{n-1}).
$$
\end{thm}
\begin{proof}
By density argument, it suffices to consider $v\in C_c^\infty (\mathbb{R}^n)$.
We write $\xi = (\xi^\prime, \xi_n) \in \mathbb{R}^n$. 
By the Fourier inversion formula we find that 
\begin{eqnarray*}
\begin{split}
\gamma_{0} v\left(x^{\prime}\right) &=\left.\int_{\mathbb{R}^{n}} e^{i 2 \pi x \cdot \xi} \hat{v}(\xi) d \xi\right|_{x_{n}=0}=\int_{\mathbb{R}^{n}} e^{i 2 \pi x^{\prime} \cdot \xi^{\prime}} \hat{v}(\xi) d \xi\\
&=\int_{\mathbb{R}^{n-1}}\left(\int_{\mathbb{R}} \hat{v}\left(\xi^{\prime}, \xi_{n}\right) d \xi_{n}\right) e^{i 2 \pi x^{\prime} \cdot \xi^{\prime}} d \xi^{\prime},
\end{split}
\end{eqnarray*}
Therefore, we have
\begin{eqnarray*}
\begin{split}
\mathcal{F}\left(\gamma_{0} v\right)\left(\xi^{\prime}\right)&=\int_{\mathbb{R}} \hat{v}\left(\xi^{\prime}, \xi_{n}\right) d \xi_{n}
\\
&=\int_{\mathbb{R}}\kappa(\xi)^{-\frac{1}{2}}\left(1+|\xi|^{2}\right)^{-s / 2}\kappa(\xi)^{\frac{1}{2}}\left(1+|\xi|^{2}\right)^{s / 2} \hat{v}\left(\xi^{\prime} ,\xi_{n}\right) d \xi_{n}.
\end{split}
\end{eqnarray*}
Applying the Cauchy-Schwarz inequality and multiplying both sides by $\kappa(\xi^{\prime},0)$ yields
\begin{eqnarray*}
\begin{split}
\kappa(\xi^\prime,0) \left|\mathcal{F}\left(\gamma_{0} v\right)\left(\xi^{\prime}\right)\right|^{2} &\leq \int_{\mathbb{R}}\frac{\kappa({\xi',0})}{\kappa({\xi})}\left(1+|\xi|^{2}\right)^{-s} d \xi_{n} \int_{\mathbb{R}}\kappa({\xi})\left(1+|\xi|^{2}\right)^{s}\left|\hat{v}\left(\xi^{\prime}, \xi_{n}\right)\right|^{2} d \xi_{n}\\
& \leq C_0 \int_{\mathbb{R}}\left(1+|\xi|^{2}\right)^{-s} d \xi_{n} \int_{\mathbb{R}}\kappa({\xi})\left(1+|\xi|^{2}\right)^{s}\left|\hat{v}\left(\xi^{\prime}, \xi_{n}\right)\right|^{2} d \xi_{n}
\end{split}
\end{eqnarray*}
Now, by the substitution $\xi_n=(1+|\xi'|^2)^{\frac{1}{2}} t$, we denote 
\begin{eqnarray*}
\begin{split}
M_s(\xi') &:= \int_{\mathbb{R}} \frac{d \xi_{n}}{(1+|\xi|^{2})^s} =\int_{\mathbb{R}} \frac{d \xi_{n}}{\left(1+\left|\xi^{\prime}\right|^{2}+\left|\xi_{n}\right|^{2}\right)^{s}}\\
&=\frac{1}{(1+\left|\xi^{\prime}\right|^{2})^{s-\frac{1}{2}}} \int_{\mathbb{R}}\frac{dt}{(1+t^2)^s}.
\end{split}
\end{eqnarray*}
The integral 
$$C_s := \int_{\mathbb{R}}\frac{dt}{(1+t^2)^s}$$ is integrable if and only if $s> \frac{1}{2}$. 
Therefore, we can bound
$$ \kappa(\xi',0) {(1+\left|\xi^{\prime}\right|^{2})^{s-\frac{1}{2}}} \left|\mathcal{F}\left(\gamma_{0} v\right)\left(\xi^{\prime}\right)\right|^{2}
\leq C_0 C_s\int_{\mathbb{R}}\kappa({\xi})\left(1+|\xi|^{2}\right)^{s}\left|\hat{v}\left(\xi^{\prime}, \xi_{n}\right)\right|^{2} d \xi_{n}.
$$
Integrating with respect to $\xi' $ yields
$$
\left\|\gamma_{0} v\right\|_{H^{s-\frac{1}{2}}_\kappa\left(\mathbb{R}^{n-1}\right)} \leq C_{0}^{\frac{1}{2}}C_s^{\frac{1}{2}}\|v\|_{H^{s}_\kappa\left(\mathbb{R}^{n}\right)}
$$
where the constant $C_s$ depends on $s$, and $C_0$ depends on the ratio of $\kappa(x)$'s boundary value to its interior value.
\end{proof}
\begin{rk}
If we assume that $C_0$ is much less than the contrast $\frac{\kappa_1}{\kappa_0}$,     then the constant in this weighted trace theorem does not depend on the contrast.  
\end{rk}

 \begin{thm}
 Let $\Omega\subset \mathbb{R}^n$ be a bounded Lipschitz domain with boundary $\Gamma$ and assume that we have a weight function $\kappa(\xi)$ defined on the $\bar{\Omega}$ and it satisfies  $$\frac{\sup_{\xi\in\Gamma}\kappa({\xi})}{\inf_{\xi\in \Omega}\kappa({\xi})}\leq C_0,$$ 
 where $C_0$ is a positive constant and $\kappa({\xi})\neq 0$ for all $\xi\in \overline{\Omega}$.

\begin{enumerate}[(a)]
    \item For $1/2<s<3/2$, $\gamma_0$ has a unique extension to a bounded linear operator
    $$
\gamma_0 : H^s_\kappa(\Omega)\to H^{s-\frac{1}{2}}_\kappa(\Gamma).
$$
    \item (Inverse Trace Theorem)
 For $1/2<s<3/2$, the trace operator $\gamma_0: H^s_\kappa(\Omega)\to H^{s-\frac{1}{2}}_\kappa(\Gamma)$ has a continuous right inverse operator
$$\mathcal{E}: H^{s-\frac{1}{2}}_\kappa(\Gamma)\to H^s_\kappa(\Omega)$$
 satisfying $(\gamma_0\circ \mathcal{E})(w)=w$ for all $w\in H^{s-\frac{1}{2}}_\kappa(\Gamma)$. To be more specific, there exists a constant $c>0$ such that 
 $$
\| \mathcal{E} w\|_{H^s_\kappa(\Omega)}\leq c \|w\|_{{H^{s-\frac{1}{2}}_\kappa}(\Gamma)}~ ~~\tforall w\in H^{s-\frac{1}{2}}_\kappa(\Gamma).
 $$
 
\end{enumerate}

\end{thm}
\begin{proof}
For part (a), without going into the details, we mention that one can uses a partition of unity and local transformation onto subsets of $\mathbb{R}^{n-1}$ to get the result, cf. L.C. Evans \cite{evans1998partial}.

For part (b), since $\gamma_0$ is continuous and surjective, the Open Mapping Theorem implies that there exists $c > 0$ such that, for all $w\in H^{s-\frac{1}{2}}_\kappa(\Gamma)$, there is $\mathcal{E}w\in H^s_\kappa(\Omega)$ satisfying $\gamma_0(\mathcal{E}w)=w$ and
 $
\| \mathcal{E} w\|_{H^s_\kappa(\Omega)}\leq c \|w\|_{{H^{s-\frac{1}{2}}_\kappa}(\Gamma)}.
 $
\end{proof}
One might notice that we are not using the usual weighted norm in the proof above. Now, we want to show that, under our assumptions about $\kappa(x)$, the usual weighted Sobolev norm 
$$
\|u\|_{{H_\kappa^s}(\Omega)}= \left ( \sum_{|\alpha|\leq s} \|D^\alpha u\|_{{L^2_\kappa}(\Omega)} \right )^{\frac{1}{2}} 
$$
is equivalent to the weighted Sobolev norm defined through Fourier transform
$$\|v\|_{H^s_\kappa(\Omega)}:= \|v_0\|_{H^s_\kappa(\mathbb{R}^n)}=\int_{\mathbb{R}^n}\kappa({\xi})\left(1+|\xi|^{2}\right)^{s}\left|\hat{v_0}\left(\xi\right)\right|^{2} d \xi$$ where $v_
0$ denotes the extension of $v$ by $0$ onto $\mathbb{R}^n\setminus \Omega$.

Let us denote the usual weighted norm as $\|\cdot\|_1$ and the weighted norm defined through Fourier transform as $\|\cdot\|_2$ to distinguish them below. 
It is sufficient to prove these two norms are equivalent on $\mathbb{R}^n$. Let $\|\cdot\|_a$ and $\|\cdot\|_b$ be two norms on $H^s_\kappa(\mathbb{R}^n)$, we denote $\|\cdot\|_a\sim \|\cdot\|_b$ if there exist two constants $c_0$ and $c_1$ independent of the contrast $\frac{\kappa_1}{\kappa_0}$ such that 
$c_0\|u\|_b\leq\|u\|_a\leq c_1\|u\|_b,~\forall u\in H^s_\kappa(\mathbb{R}^n)$.
\begin{thm}
$\|\cdot\|_1$ and $\|\cdot\|_2$ are equivalent.
\end{thm}
\begin{proof}
Since we assume that $\kappa(x)$ is piecewise constant with $0<\kappa_0\leq\kappa(x)\leq \kappa_1$ and $\kappa_0$ is a relatively small number, we can derive from the Plancerel theorem that 
$\|f\|_{L^2_\kappa(\mathbb{R}^n)}\sim\|\hat{f}\|_{L^2_\kappa(\mathbb{R}^n)}$. Hence, we have 


$$
\begin{aligned}
\|u\|_1^2&=\sum_{|\alpha| \leqslant s}\left\|D^{\alpha} u\right\|_{L^{2}_\kappa (\mathbb{R}^n)}^{2}
\sim\sum_{|\alpha| \leqslant s}\left\|\widehat{D^{\alpha} u}\right\|_{L^{2}_\kappa (\mathbb{R}^n)}^{2}\\
&=\sum_{|\alpha| \leqslant s}\left\|\xi^\alpha\hat{ u}\right\|_{L^{2}_\kappa (\mathbb{R}^n)}^{2}=\int_{\mathbb{R}^n} \sum_{|\alpha|\leq s}|\xi^\alpha|^2|\hat{ u}|^2\kappa(\xi)  d\xi.
\end{aligned}
$$

On the other hand,
$$
\|u\|_2^2= \int_{\mathbb{R}^n}\kappa({\xi})\left(1+|\xi|^{2}\right)^{s}\left|\hat{u}\left(\xi\right)\right|^{2} d \xi.
$$
To prove $\|\cdot\|_{1}$ is equivalent to $\|\cdot\|_2$, we only need to prove $F(\xi)=(1+|\xi|^2)^s$ and $S(\xi)= \sum_{|\alpha|\leq s}|\xi^\alpha|^2$ can bound each other. If we expand $F(\xi)$ by binomial expansion, we can see that it has the same terms as $S(\xi)$ but only with different coefficients, which are simply binomial coefficients. Therefore, we can replace the coefficients of $F(\xi)$ by a suitable number to bound $S(\xi)$ from either side.
\end{proof}
\section{Stable extension operator on multiscale space}\label{sec:StableExtension}
In this section, we construct a stable extension operator from the finite element space $W_H(\Gamma_C)$ to the multiscale space $X_H(\Omega) $.
\noindent

Let us define the extension operator $L_H: W_H(\Gamma_C)\to X_H(\Omega)$ as 
\begin{equation}\label{b5.1}
L_H(w_H)(x)=\sum_{i=1}^{N_0} b_i v_{ex1}^{(i)}(x)+ I_H(L_h(w_H))
\end{equation}
where $\{b_i\}_{i=1}^{N_0}$ satisfy 
\begin{equation}\label{b_ieqn}
\int_{\Gamma_C} \kappa\left(\sum_{i=1}^{N_0} b_i  v_{ex1}^{(i)} (x)- w_H(x)\right) \psi_H(x) ~d\Gamma= 0 
\end{equation} for all $\psi_H(x) \in W_H(\Gamma_C)$ and the projection operator $I_H$ is defined in Lemma \ref{hlemma4.5Ms}. Since $W_H(\Gamma_C)$ is a space spanned by the set  $\{\psi_k\}_{k=1}^{N_0}$, the vector $[b_1,b_2,\cdots, b_{N_0}]$ is the solution to the system of equations defined by \eqref{b_ieqn}.

\begin{thm}\label{bThm5.1}
Assume that $(\mathcal{T}^H)_H $  is a regular family of triangulations or quadrangulations of $\overline{\Omega}$, and let $L_H$ be defined by \eqref{b5.1}. There exists a constant $C$, independent of $H$, such that
\begin{equation}\label{b5.2}
 \|L_H(w_H)\|_{H^1_\kappa(\Omega)}\leq C\|w_H\|_{H^{\frac{1}{2}}_\kappa(\Gamma_C)}~~~ \tforall w_H\in W_H(\Gamma_C)
\end{equation}
Moreover, the trace of $L_H(w_H)$ on $\Gamma_C$ satisfies
$$\int_{\Gamma_C}\kappa (L_H(w_H) -w_H) \psi_H ~d\Gamma =0.$$ 
\end{thm}
\begin{proof}
For any $w_H\in W_H(\Gamma_C)$, we have
\begin{equation}\label{b5.3}
\|L_H(w_H)\|_{H^1_\kappa(\Omega)}\leq \norm{\sum_{i=1}^{N_0} b_i  v_{ex1}^{(i)}}_{H^1_\kappa(\Omega)}+\left\|I_H(L_h(w_H))\right\|_{H^1_\kappa(\Omega)}.
\end{equation}
The second term is estimated by Theorems \ref{thm1.1} and the fine-scale lifting operator $L_h$ mentioned in \cite{bernardi1998local},
\begin{equation}\label{b5.4}
\left\|I_{H}\left(L_h\left(w_{H}\right)\right)\right\|_{H^1_\kappa(\Omega)} \leq C_{1}\left\|L_h\left(w_{H}\right)\right\|_{H^1_\kappa(\Omega)} \leq C_{2}\left\|w_{H}\right\|_{H^{\frac{1}{2}}_\kappa\left(\Gamma_{C}\right)}.
\end{equation}

Let $K$ be an element of $\mathcal{T}^H$ such that some nodes of $K$ lies on $\Gamma_C$. Suppose that \(\omega_{K}\) contains $n$ macroelements \(\omega_{i}\). We agree to number first, say from $1$ to $n_0$, the nodes $\boldsymbol{a_i}$ that lie on $\Gamma_C$ and from $n_0+1$ to $n$ the remaining nodes.

Then, owing to the support of the basis functions $v^{(i)}_{ex1} (x)$, we have
\begin{equation}\label{firstterm}
\left\|\sum_{i=1}^{N_0} b_i v_{ex1}^{(i)}(x)\right\|_{H^1_\kappa(K)} \leq \sum_{i=1}^{n_0}|b_i|\|v^{(i)}_{ex1} (x)\|_{H^1_\kappa(K)}.
\end{equation}

On one hand, 
\begin{equation}\label{b5.5}
\left\|v^{(i)}_{ex1}\right\|_{H^1_\kappa(K)} \lesssim \norm{\delta_i}_{L^2_\kappa(\partial{\omega_K})}
\end{equation}
since $v^{(i)}_{ex1}$ is the solution of a Dirichlet problem \eqref{ex1}.
On the other hand, $[b_1,b_2,\cdots, b_{N_0}]$ is the solution to the system of equations defined by \eqref{b_ieqn}. Then we have 
\begin{equation}\label{b_iestimate}
    |b_i|\leq C\|w_H\|_{L^2_\kappa(\Gamma_C)} \leq C\|w_H\|_{H^{\frac{1}{2}}_\kappa(\Gamma_C)}.
\end{equation}

Then \eqref{b5.2} follows from \eqref{b5.3}, \eqref{b5.4}, \eqref{firstterm}, \eqref{b5.5}, and  \eqref{b_iestimate} together with Theorems \ref{thm1.1}.
\end{proof}

\section{Weighted norm on the boundary in 3D}\label{3dbdnorm}
Let $\partial \Omega$ denote a Lipschitz domain, and let $\{\omega_l\}_{l=1}^K$ denote an open covering of $\partial \Omega$, such that for each $l\in \{1,2,\cdots, K\}$, there exist bi-Lipschitz charts which satisfy 
$$\vartheta_{l}: B(0, r_l)\subset \mathbb{R}^2\rightarrow \omega_l ~~~\text{ is a bi-Lipschitz homeomorphism} $$
Next, for each $1\leq l\leq K$, let $0\leq \varphi_l\in C^\infty_0(\omega_l)$ denote a partition of unity so that $\sum_{l=1}^K \varphi_l (x) =1$ for all $x\in \partial \Omega$. We define 
$$
H^1_\kappa(\partial \Omega)=\{ u\in L^2(\partial \Omega) : \|u\|_{H^1_\kappa(\partial \Omega)}<\infty\},
$$
where for all $u\in H^1_\kappa(\partial\Omega),$
$$
\|u\|_{H^1_{\kappa}(\partial \Omega)}^2=\sum_{l=1}^K \| (\phi_l u)\circ \vartheta_l \|^2_{H^1_{\tilde{\kappa}_l}(\mathbb{R}^2)}.
$$
Here, the weight $\tilde{\kappa}_l$ is defined as $(\varphi_l \kappa)\circ \vartheta_l$.
\end{appendices}
\end{document}